\documentclass[onefignum,onetabnum]{siamart220329}

\usepackage{geometry}
\usepackage{amsmath}
\usepackage{enumerate}
\usepackage{amssymb}
\usepackage{graphicx}
\usepackage{subcaption}
\usepackage{caption}
\usepackage{bbm}
\usepackage{enumitem}
\usepackage{breqn}
\usepackage{thm-restate}
\usepackage{cleveref}

\newcommand\abs[1]{\left\lvert{#1}\right\rvert}

\newcommand\norm[1]{\left\lVert{#1}\right\rVert}

\def\R{\mathbb{R}}

\def\sA{\mathcal{A}}
\def\sF{\mathcal{F}}

\def\sU{\mathcal{U}}

\def\sE{\mathcal{E}}

\def\eps{\varepsilon}

\def\E{\mathbb{E}}
\def\P{\mathbb{P}}

\def\one{\mathbbm{1}}

\mathcode`l="8000
\begingroup
\makeatletter
\lccode`\~=`\l
\DeclareMathSymbol{\lsb@l}{\mathalpha}{letters}{`l}
\lowercase{\gdef~{\ifnum\the\mathgroup=\m@ne \ell \else \lsb@l \fi}}%
\endgroup

\newtheorem{defn}{Definition}
\newtheorem{assumption}{Assumption}

\newtheorem{prop}{Proposition}

\theoremstyle{plain}

\newenvironment{remark}{\noindent{\bf Remark}\hspace*{1em}}{\bigskip}

\usepackage[utf8]{inputenc}
\usepackage{xcolor}
\usepackage{booktabs}
\usepackage{arydshln}
\usepackage{algorithm}
\usepackage[noend]{algpseudocode}
\def\Var{\text{Var}}

\newcommand{\rev}[1]{{\color{blue}{{#1}}}}
\renewcommand{\rev}[1]{{#1}}

\newcommand{\revv}[1]{{\color{blue}{{#1}}}}
\renewcommand{\revv}[1]{{#1}}

\newcommand*\samethanks[1][\value{footnote}]{\footnotemark[#1]}

\author{Billy Jin\thanks{School of Operations Research and Information Engineering, Cornell University, Ithaca, NY, USA. Email: \texttt{\{bzj3,ks2375,mx229\}@cornell.edu}.} \and Katya Scheinberg\samethanks[1] \and Miaolan Xie\samethanks[1]}

\title{High Probability Complexity Bounds for Adaptive Step Search Based on Stochastic Oracles}

\begin{document}
	
	\maketitle
	
	\begin{abstract}
		We consider a  {\em step search method}\footnote{We introduce the term {\em step search} for a class of methods, similar to line search, but where step direction can change during the back-tracking procedure.} for continuous optimization under a stochastic setting where the function values and gradients are available only through inexact probabilistic zeroth- and first-order oracles.
		Unlike the stochastic gradient method and its many variants, the algorithm does not use a pre-specified sequence of step sizes but increases or decreases the step size adaptively according to the estimated progress of the algorithm. These oracles capture multiple standard settings including expected loss minimization and zeroth-order optimization. Moreover, our framework is very general and allows the function and gradient estimates to be biased.  The proposed algorithm is simple to describe and easy to implement.  Under fairly general conditions on the oracles, we derive a high probability tail bound on the iteration complexity of the algorithm when it is applied to non-convex, convex, and strongly convex \rev{(more generally, those satisfying the PL condition)} functions. Our analysis strengthens and extends prior results for  stochastic step  and line search methods.   
	\end{abstract}
	
	% REQUIRED

% REQUIRE

%\begin{MSCcodes}
%	...
%\end{MSCcodes}

	\section{Introduction}
	
		In this paper, we analyze  a {\em step search} method for minimization of an unconstrained, differentiable, possibly non-convex function $\phi: \R^n \to \R$.  
%		The goal is to find a $\varepsilon$-stationary point for $\phi$; \ml{remove this sentence?} that is, a point $x$ with $\norm{\nabla \phi(x)} \leq \varepsilon$. 
		 We make the standard assumption that $\nabla \phi$ is $L$-Lipschitz, but the knowledge of $L$ is not assumed by the algorithm.  We consider a setting where
	 neither the function value $\phi(x)$ nor  the gradient $\nabla \phi(x)$ are directly computable. Instead, the algorithm is given black-box access to the following probabilistic  oracles:
	\begin{itemize}
	\item {\bf Stochastic zeroth-order oracle (SZO($\epsilon_f, \nu, b$))}
Given  a point $x$,  the oracle computes  $f(x,\xi)$, a (random) estimate of the function value $\phi(x)$. $\xi$ is a random variable (whose distribution may depend on $x$), with probability space $(\Omega, \mathcal{F}_{\Omega}, P)$.  We assume the absolute value of the estimation error $e(x)=|f(x,\xi(x))-\phi(x)|$ (we omit the dependence on $\xi$ for brevity) to be a ``one-sided" subexponential-like random variable\footnote{This is a weaker requirement than assuming $e(x)$ to be subexponential and is sufficient for our purposes. } with parameters 
$(\nu,b)$, whose mean is bounded by some constant $\epsilon_f>0$. Specifically, 
	\begin{equation}\label{eq:zero_order}
	{\mathbb E_\xi}\left [ e(x)\right ]\leq \epsilon_f \ \text{and }\ {\mathbb E_\xi}\left [\exp\{\lambda (e(x)-\E[e(x)])\}\right ]\leq \exp\left(\frac{\lambda^2\nu^2}{2}\right), \quad \forall \lambda\in \left[0,\frac{1}{b}\right].
	 \end{equation}

In summary, the input to the oracle is $x$, the output is $f(x,\xi)$, and the values $(\epsilon_f, \nu, b)$ are intrinsic to the oracle. 
%\item {\bf Probabilistic first order oracle.}	 
%	  Given a point $x$ and a constant $\alpha> 0$, the oracle computes $g(x,\xi^\prime)$, a (random) estimate of the gradient $\nabla \phi(x)$, such that
%\begin{equation}\label{eq:first_order}
%{\mathbb P_{\xi^\prime}\left (\|g(x,\xi^\prime)-\nabla \phi(x)\|\leq \max \{\epsilon_g, \kappa \alpha\|g(x,\xi^\prime)\| \}\right)\geq 1-\delta.}
%\end{equation}
%\bjcomment{should $\delta$ be intrinsic to the oracle?}
\item {\bf Stochastic first-order oracle (SFO($\epsilon_g, \tau, \kappa,\delta$)).}	 
Given a point $x$ and a constant $\alpha> 0$, the oracle computes $g(x,\xi^\prime)$, a (random) estimate of the gradient $\nabla \phi(x)$, such that
\begin{equation}\label{eq:first_order}
	{\mathbb P_{\xi^\prime}\left (\|g(x,\xi^\prime)-\nabla \phi(x)\|\leq \max \{\epsilon_g, \min\{\tau,\kappa \alpha\}\|g(x,\xi^\prime)\| \}\right)\geq 1-\delta.}
\end{equation}
In summary, the input to the oracle is $x$ and $\alpha$,  the output is $g(x,\xi^\prime)$, and the values $(\epsilon_g, \tau,\kappa,\delta)$ are intrinsic to the oracle. 

\end{itemize}
%\begin{remark}
%We will analyze a step search algorithm that relies on these two oracles. In the zeroth order oracle, the constants $\epsilon_f$ and $(\nu, b)$ are intrinsic. In the first order oracle, $\kappa, \epsilon_g$, and $\delta$ are intrinsic. These values cannot be controlled. On the other hand, $\alpha$ is an \emph{input} to the first order oracle that can be chosen by the algorithm. In fact, as we shall see in Section \ref{sec:alg}, $\alpha$ will be the step size of the step search method. 
%\end{remark}

 These two oracles cover several settings, including
 \begin{itemize}
 \item  Standard supervised learning, where gradients and values of the loss function are computed based on a mini-batch. Here, the random variables $\xi$ and $\xi^\prime$ in the zeroth- and first-order oracles represent the random set of samples in the mini-batch. 
 \item Zeroth-order optimization, where gradients are estimated via randomized finite differences using (possibly noisy) function values. This arises in policy gradients in reinforcement learning, as is used in \cite{salimans2016evolution} and analyzed in \cite{berahas2019theoretical}.
 \item A variety of other settings, where the gradients and function estimates may be biased stochastic estimates of the true gradients and function values. 
 \end{itemize}
 
 Let us explain the choice of the oracle definitions here.  SFO($0,\infty, \kappa,\delta$) (i.e., the first-order oracle with  $\epsilon_g=0$ and $\tau=\infty$) was used in  \cite{CS17}  and \cite{berahas2019global} in the analysis of a stochastic step search method.\footnote{These prior papers referred to the methods as line search, however, unlike traditional line search \cite{NW} these methods choose a new stochastic gradient estimate, and thus a new search direction,  at each back-tracking step. In this paper we propose to refer to such methods as step search methods.}
 It has been discussed in those works how the oracle compares to the standard unbiased stochastic gradient estimator. In general, SFO($0,\infty, \kappa,\delta$) can produce gradient estimates with arbitrarily large bias and variance because the error $\|g(x,\xi^\prime)-\nabla \phi(x)\|$ is only bounded with some given probability. \rev{On the other hand, if  $\epsilon_g=0$, then  the oracle needs to be able to  produce an estimate $g(x,\xi^\prime)$ with an arbitrary small error (with the given probability), in the case when $\alpha$ is small. Thus   $\epsilon_g>0$ accounts for the practical limit of the  oracle in terms of accuracy (e.g. coming from the largest allowable mini-batch size). On the other hand, using $\tau=\infty$ allows the bound to become arbitrarily loose if $\alpha$ is large. This does not represent practice and thus we allow $\tau<\infty$.}

  \rev{Our choice of the zeroth-order oracle may appear somewhat convoluted at first. However, as we discuss here, it strikes a natural balance between theory and practice, and allows us to improve on prior work. 
 Specifically, in \cite{CS17}  the zeroth-order oracle was assumed to be exact, which corresponds to the case $\epsilon_f=0$  in our zeroth-oracle definition, with arbitrary $\nu$ and $b$. 
  In \cite{berahas2019global}  the zeroth-order oracle requirements were relaxed, compared to those in  \cite{CS17}, allowing the error to be deterministically bounded by some $\epsilon_f\geq 0$, which is a special case of our zeroth-order oracle with $\nu=\epsilon_f$ and $b=0$. 
 In  \cite{paquette2018stochastic}, a more general stochastic zeroth-order oracle  was used   within a stochastic step search method where the error is allowed to be arbitrarily large, but with sufficiently small variance.  This oracle condition, while relatively loose,  necessitated a more complicated and somewhat more constrained first-order oracle, and a somewhat artificial modification to the method itself.  The analysis of this method is also very different from those used in   \cite{CS17}  and \cite{berahas2019global}. 
 
 Our zeroth-order oracle essentially considers errors in the function that are not deterministically bounded, but are light tailed. Thus our framework covers a much broader class of stochastic settings than those covered in \cite{berahas2019global}. \revv{For example, in empirical risk minimization in machine learning, the assumption that the function error is light tailed holds via Hoeffding's inequality if the loss function is bounded, or if the support of the data set is bounded, the loss is Lipschitz, and the set of decision variables we consider is bounded.}
% 	Note that this error bound can be much better than simply assuming that the function estimate error itself is bounded, since averaging many iid samples gives concentration.
In other settings, one can argue that if the zeroth-order oracle is implemented via the averaging of random estimates over a relatively large number of samples, then the light tail behavior is justified by the Central Limit Theorem. 
 
 Finally, we would like to point out that the following alternative zeroth-oracle definition is equivalent to the one above.
 
 	{\bf Stochastic zeroth-order oracle, alternative definition} ($\mathsf{SZO^\prime}(\epsilon_f, \lambda,a)$). 
	Given  a point $x$,  the oracle computes  $f(x,\xi)$, 
	where $\xi$ is a random variable, whose distribution may depend on $x$, $\epsilon_f, \lambda$ and $a$, that satisfies 
	\begin{equation*}
	{\mathbb E_{\xi}}\left [ |{\phi(x) - f(x, \xi)}|\right ]\leq \epsilon_f \ \text{~and~ }
	{\mathbb P_{\xi}}\left ( |{\phi(x) - f(x, \xi)}|< t\right ) \geq 1-e^{\lambda (a-t)}, 
	\end{equation*}
	for any $t>0$. 
	It can be shown that having $\mathsf{SZO^\prime}(\epsilon_f, \lambda,a)$ implies also having $\mathsf{SZO}(\epsilon_f, \nu,b)$ for some constants $\nu$ and $b$ whose values depend on $\lambda$, $a$ and some universal constants (see , e.g. Proposition 2.7.1 of \cite{vershynin2018high}). Similarly, having $\mathsf{SZO}(\epsilon_f, \nu,b)$ also implies having $\mathsf{SZO^\prime}(\epsilon_f, \lambda,a)$. 
	Since  our results explicitly depend on $\nu$ and $b$, we choose to use $\mathsf{SZO}(\epsilon_f, \nu,b)$ as our zeroth-order oracle, even though $\mathsf{SZO^\prime}(\epsilon_f, \lambda,a)$ may appear to be more intuitive.}
 
 In all three prior papers on the stochastic step search method \cite{CS17,paquette2018stochastic, berahas2019global}, the  expected complexity is shown to be comparable to that of deterministic line search, in terms of its dependence of the desired convergence accuracy $\varepsilon$. The dependence on other constants, such as the Lipschitz constant of the gradient of $\phi(x)$, 
 is worse for the expected complexity in \cite{paquette2018stochastic} vs. the results in  \cite{CS17}  and \cite{berahas2019global}. In addition, because $\tau$ was chosen to be $\infty$ in those prior works, it was necessary to impose an upper bound on the step size parameter, with this value explicitly appearing in the complexity. By introducing a finite value of $\tau$, we simply constrain the worst case accuracy of the first-order oracle - a very minor condition which allows us to drop the upper bound on the step size parameter completely.  In summary, our choices of zeroth- and first-order oracles here are dictated by the key  motivations of this paper:
 \begin{itemize}
 \item Extend the analysis in \cite{berahas2019global} to the more general case where $e(x)$ is an unbounded random variable.  
 \item Use a more relaxed form of the first-order oracle than in the previous stochastic step search papers, by allowing $\epsilon_g\neq 0$.
 \item Remove the upper bound on the step size parameter from the stochastic step search method by allowing $\tau<\infty$. 
 \item Derive a simpler and stronger analysis for a simple stochastic step search method,  compared to those in \cite{paquette2018stochastic}, under somewhat stronger, but natural conditions on the zeroth-order oracles.
 \item Derive a {\em high probability} tail bound instead of only a bound on the expected complexity, showing that the complexity itself is a subexponential random variable, using ideas from \cite{gratton2018complexity}.  
 \end{itemize}

%\blue{Here we want to note that some alternative first order oracle conditions can be used: 
%\begin{equation}\label{eq:new_first_order}
%	\mathbb P_{\xi^\prime}\left (\|g(x,\xi^\prime)-\nabla \phi(x)\|\leq \max \{\epsilon_g, \kappa\alpha\|g(x,\xi^\prime)\| \}\right)\geq 1-\delta,
%\end{equation}
%or
%\begin{equation}
%	\mathbb P_{\xi^\prime}\left (\|g(x,\xi^\prime)-\nabla \phi(x)\|\leq \max \{\epsilon_g, {\tau}\|g(x,\x. i^\prime)\| \}\right)\geq 1-\delta.
%\end{equation}}
All the theory in this paper can be carried out for $\tau=\infty$. In addition, if $\tau<1$, then with small modifications the theory can be carried out  for the first-order oracle where $\min\{\tau, \kappa \alpha\}$ is replaced simply by $\tau$. However we chose to focus on condition 
\eqref{eq:first_order}, as it allows for a more adaptive oracle and is closer to what was used in prior literature.

In addition to the works mentioned above, an extension of the traditional line search  for stochastic empirical loss minimization is analyzed in  \cite{vaswanie2019painless}, where the function oracles are implemented using a random mini-batch of a fixed size. Unlike in the step search algorithm, the mini-batch remains fixed during backtracking until a standard Armijo condition is satisfied  \cite{NW}.  Thus the search direction remains the same until a step is taken just like in the standard Armijo line search. While good computational performance has been reported in \cite{vaswanie2019painless}, its theoretical analysis requires several very restrictive assumptions. Specifically, it is assumed that $f(x,\xi)$ is Lipschitz smooth for any realization of $\xi$ and that in fact for every realization $g(x,\xi)=\nabla f(x,\xi)$. This is a very strong assumption that often fails in practice. 
In addition, the analysis for non-convex $\phi(x)$ is carried out only under severe upper bounds on the step size parameter, rendering the line search essentially impractical.   
In contrast, our theoretical results are stronger than those in  \cite{vaswanie2019painless}, and we only assume Lipschitz smoothness of $\phi(x)$. 

\revv{An earlier version of the work was previously published in the NeurIPS conference \cite{jin2021high}}. It contains the analysis framework that applies to the non-convex setting and requires an upper bound of the step size parameter that affects the complexity bound. 
Here we extend the framework to cover the convex and strongly convex cases, as well as remove the upper bound on the step size parameter. 
In addition, here we remove a strong assumption on the independence of the errors in the function estimates.

{\bf  In summary}, we present an analysis of an adaptive step search algorithm  under very general conditions on the gradient and function estimates for non-convex, convex and strongly convex functions. \rev{In fact, the results for the strongly convex case also hold in the more general setting where only the PL condition is assumed.} The results not only subsume most results in the prior literature, but also substantially extend the framework. Moreover, high probability tail bounds on iteration complexity are derived, instead of only expected iteration complexity.

The paper is organized as follows. 
  The step search algorithm is presented  in Section \ref{sec:alg}.  	The algorithm generates a stochastic process with certain properties, which are stated as assumptions in Section \ref{sec:high_prob} and are shown to imply the main complexity bound. In Section \ref{three_setting}, we show how these properties in fact hold when the algorithm is applied to non-convex, convex and strongly convex $\phi(x)$.
  In Section \ref{oracles},  we present a brief discussion on the oracles in  two practical settings. Computational experiments on empirical risk minimization are presented in Section \ref{sec:experiments}, and the final conclusions are in Section \ref{sec:conclusion}.

	\section{Stochastic adaptive step search algorithm and related notations}
	\label{sec:alg}
	
	In this paper we impose the following standard assumption on $\phi(x)$.	
	% 		\subsection{Line Search Algorithm}
	\begin{assumption}\label{ass:Lipschitz}
		$\nabla \phi$ is $L$-Lipschitz continuous and $\phi$ is bounded from below by some constant $\phi^*$. 
	\end{assumption}

		We consider the step search algorithm proposed by \cite{berahas2019global}, which is an extension of the step search algorithm in \cite{CS17} to the setting of inexact function estimates. In both algorithms, a random gradient estimate is used to attempt a step. Compared to \cite{CS17}, the key modification of the algorithm in  \cite{berahas2019global} is the relaxation of the Armijo condition by an additive constant $2\epsilon_f^\prime$. The difference between this algorithm and the more standard line search methods such as the ones in  \cite{NW} and  \cite{vaswanie2019painless} is that 
		the gradient estimate is recomputed at each iteration, whether or not a step is accepted.  
				
		The algorithm is presented below.\footnote{A similar algorithm was named ALOE  in the earlier version of this paper, before we decided to adopt the step search vs. line search terminology.}
		{
	\begin{algorithm}[ht]
		\caption{~\textbf{Stochastic Adaptive Step Search (SASS)}}
		\label{alg:ls}
		{\bf Input:}  Oracles SZO($\epsilon_f, \nu, b$) and  SFO($\epsilon_g, \tau, \kappa,\delta$), starting point $x_0$, initial step size $\alpha_0$,  constants $\theta, \gamma \in (0,1)$ and $\epsilon_f^\prime \geq 0$.  
		%   and probability $1-\delta\in (0,1)$ for the first order oracle. 
		\begin{algorithmic}[1]
			\For{$k=0,1,2,\dots$}
			\State  {\bf Compute gradient approximation ${g_k}$:} \label{step:grad_approx}
			
			\vspace{1mm}
			
			\hspace{0.5cm} \begin{minipage}{\textwidth-3cm}Generate the  direction $g_k=g(x_k,\xi_k^\prime)$ using the $(\epsilon_g, \tau, \kappa,\delta)$ probabilistic first-order oracle, with $\alpha=\alpha_k$. 
			\end{minipage}
			\vspace{1mm}
			
			\State {\bf Check sufficient decrease:} 
			
			\vspace{1mm}
			
			\hspace{0.5cm} \begin{minipage}{\textwidth-3cm}Let $x_k^+=x_k - \alpha_k g_k$. Generate 
				$f(x_k,\xi_k)$ and $ f(x_k^+,\xi_k^+)$ using the $(\epsilon_f, \nu, b)$ probabilistic zeroth-order oracle. 
%				with $\xi_k=\xi(x_k)$ and $\xi_k^+=\xi(x_k^+)$.   
				Check the modified {\em Armijo} condition:
			\end{minipage}
			\begin{equation}
				\label{eq:suff_decr}
				f(x_k^+,\xi_k^+) \leq f(x_k,\xi_k) - \alpha_k\theta\norm{g_k}^2 + 2\epsilon_f^\prime.
			\end{equation}
			
			\State {\bf Successful step:} 
			
			\vspace{0.5mm}
			
			\hspace{0.5cm}  If (\ref{eq:suff_decr}) holds, then set $x_{k+1} \gets x_k^+ $ and $\alpha_{k+1} \gets \gamma^{-1}\alpha_k$.
			
			\vspace{0.5mm}
			
			\State {\bf Unsuccessful step:} 
			
			\vspace{0.5mm}
			
			\hspace{0.5cm} Otherwise, set $x_{k+1} \gets x_k$ and $\alpha_{k+1} \gets \gamma \alpha_k$. 
			\EndFor
		\end{algorithmic}
	\end{algorithm}
}
	
It is important to note, that while stochastic oracles SZO and SFO are inputs  to the algorithm (since the algorithm uses them to compute function and gradient estimates) the
values intrinsic to the oracles, $\epsilon_f, \nu, b,\epsilon_g, \tau, \kappa$ and $\delta$ are not known to the algorithm.
As we will see later, these values will affect the convergence properties of the algorithm, but in principle the algorithm is implemented without knowing the values. 
 In contrast all input parameters are user controlled. 
The input $\epsilon_f^\prime$ here is only required to be some upper bound for $\E{[e(x)]}$, not necessarily the tightest one.
We have the following assumption on $\epsilon_f^\prime$.
\begin{assumption}\label{epsilonf}
	$\epsilon_f^\prime \geq \epsilon_f$.
\end{assumption}
	Our computational results in Section \ref{sec:experiments} indicate that estimating $\epsilon_f^\prime$ is relatively easy in practice.  
	\rev{We further note that although the definition of the SZO requires the noise in the individual function estimates to be subexponential, in fact we only need the noise in the difference $\abs{f(x,\xi) - f(x^+,\xi^+) -(\phi(x) - \phi(x^+)) }$ to be subexponential, since it is really the noise of the difference that needs to be controlled in the analysis.}

	Due to the random nature of the function oracles, the algorithm generates a stochastic process. In the next section we  present an analysis that derives a high probability bound on the stopping time of this process, using key properties satisfied by the algorithm. We discuss these properties  in the next section, based on the notation introduced below. 
	
%	 For easier exposition we choose first derive the results under the assumption that the function noise is uniformly bounded by $\epsilon_f$ and then relax this assumption. Specifically, let $e_{k}=| f(x_k,\xi_k)-\phi(x_k)|$ and $e^+_{k}= |f(x^+_k,\xi^+_k)-\phi(x^+_k)|$. Recall that $e_k$ and $e^+_k$ satisfy \eqref{eq:zero_order} from the definition of the zeroth order oracle. We will consider two cases; 1) $e_{k}$ and $e_{k}^+$ are deterministically bounded by $\epsilon_f$, in which case $\nu$ and $b$ in \eqref{eq:zero_order} can be chosen to be $0$, and 2) $\nu$ and $b$ are not necessarily zero, in which case we assume the random variables $e_{k} + e_{k}^+$ are all independent. 
%	\begin{assumption} \label{ass:indep_noise}
%		% The random variables ${e_0+e_0^+,e_1+e_1^+,\dots}$ are either deterministically bounded by  $\epsilon_f$ or are independent.
%		Either ${e_0,e_0^+,e_1,e_1^+,\dots}$ are all deterministically bounded by  $\epsilon_f$, or the random variables $\{e_0+e_0^+, e_1+e_1^+, \ldots\}$ are independent.
%	\end{assumption}
	
A key concept that will be used in the analysis in the concept of a {\em true iteration}. 
	
	\begin{defn}[True iteration]
		\label{def:true_noisy}
		We say that iteration $k$ is \textbf{true} if  
		$$
		\|g_k-\nabla \phi(x_k)\|\leq \max \{\epsilon_g, \min\{\tau,\kappa \alpha_k\} \|g_k\|\}\quad \text{and}\quad | f(x_k,\xi_k)-\phi(x_k)| + |f(x^+_k,\xi^+_k)-\phi(x^+_k)| \leq 2\epsilon_f^\prime,$$
		and is \textbf{false} otherwise. 
	\end{defn}
	
	Let $M_k$ denote the triple $\{\Xi_k, \Xi_k^+, \Xi_k^\prime\}$, whose realizations are  $\{\xi_k, \xi_k^+, \xi_k^\prime\}$. Algorithm \ref{alg:ls} 	generates {a stochastic process $\{({G}_k, {E}_k, {E}_k^+, X_k, A_k)\}$ 
	with realizations $(g_k, e_k, e_k^+, x_k, \alpha_k)$ adapted to the filtration $\{{\cal F}_k:\, k\geq 0\}$, where ${\cal F}_k=\sigma (M_0, M_1, \ldots, M_k)$. At iteration $k$, $G_k$ is the random gradient, $E_k, E_k^+$ are the random noises in absolute value of the zeroth-order oracle at $x_k$ and $x_k^+$, $X_k$ is the random iteration point at step $k$ and $A_k$ is the random step size. 
	Note that ${G}_k$ are dictated by $\Xi_k^\prime$ in the first-order oracle, and ${E}_k, {E}_k^+$ is dictated by $\Xi_k, \Xi_k^+$ in the zeroth-order oracle.}
	We define the following random variables, measurable with respect to ${\cal F}_k$. 
		\begin{itemize}
		%			\item $X_k$ is the random variable that represents the iterate at step $k$, and its realization is denoted by $x_k$, 
		%		\item $\sA_k$ is the random variable denoting the size of step $k$, and the realization of this random variable is denoted as $\alpha_k$. 
		%		\item $G_k$ is the random variable that represents the model gradient at step $k$, and the realization of this random variable is denoted by $g_k$, 
		\item $I_k := \one\{\text{iteration $k$ is true}\}.$ % (according to Definition \ref{def:true_noisy})
		\item $\Theta_k := \one\{\text{iteration $k$ is successful}\}$.
%		\item $T_\varepsilon:=\min\{k: \norm{\nabla \phi(x_k)} \leq \varepsilon\}$, the iteration complexity of the algorithm for reaching $\varepsilon$-stationarity.
%		\item $Z_k:=\phi(x_k) - \phi^*\geq 0$, a measure of progress. 
	\end{itemize}
	%	$$\varepsilon \geq \max\left\{4\epsilon_g,\frac{1}{1+\kappa\alpha_{\mathrm{max}}}  \sqrt{\frac{2\epsilon_f}{ {L + 2\kappa}}}\right\}.$$  
		
		Next, we define the stopping time for the stochastic process generated by the algorithm, which is the quantity we want to bound. 
	
		\begin{defn}[Stopping time]
%		\ml{In later part add the assumption for }
		
%		\ml{Need to add these assumption in our main proofs? looks like no need- will check again}
		\begin{itemize}
%			Denote the minimum value of $\phi$ as $\phi^*$ .
			\item If {$\phi$} is non-convex: 
			For $\varepsilon > 0$,	$T_\varepsilon:=\min\{k: \norm{\nabla \phi(x_k)} \leq \varepsilon\}$, the iteration complexity of the algorithm for reaching $\varepsilon$-stationary point.
%			$T_\varepsilon$ is  the number of iterations required until $\| \nabla \phi(X_k)\| \leq \varepsilon$ occurs for the first time.  
			\item If {$\phi$} is strongly convex: For $\varepsilon > 0$, 	$T_\varepsilon:=\min\{k: \phi(X_k) - \phi^\star \leq \varepsilon\}$,   the number of iterations until $\phi(X_k) - \phi^\star\ \leq \varepsilon$ 
			occurs for the first time. Here, $\phi^\star = \phi(x^\star)$, where $x^\star$ is a global minimizer of {$\phi$}.
			\item If $\phi$ is convex, let $\varepsilon=(\varepsilon_0,\varepsilon_1)$ with $\varepsilon_0,\varepsilon_1>0$, $T_\varepsilon:=\min\{k: \phi(X_k) - \phi^\star \leq \varepsilon_0\ \text{or\ }
			\norm{\nabla \phi(x_k)} \leq \varepsilon_1 \}$, is the number of iterations until either $\phi(X_k) - \phi^\star\ \leq \varepsilon_0$ or $\norm{\nabla \phi(x_k)} \leq \varepsilon_1$ occurs for the first time.
	
%			 \ks{This notation needs to be fixed}. 
%			Here, $\phi^*$ is the minimum value of $\phi$.
		\end{itemize}
		We will refer to $T_\varepsilon$ as the \emph{stopping time} of the algorithm.
	\end{defn}
	
	It is easy to see that $T_\varepsilon$  is a \emph{stopping time} of the stochastic process with respect to ${\cal F}_k$. \revv{Moreover, note that the stopping time for the convex case is in terms of both the optimality gap and the gradient norm. This is because in the convex case (unlike the strongly convex case), the optimality gap does not provide a lower bound on the gradient norm. Thus, if the gradient becomes too small, due to the bias, the first-order oracle can always provide gradient estimates in the opposite direction to the true gradient, and hence make it impossible for the algorithm to progress any further.}
	
%	Moreover, note that for general convex functions, we measure the progress using the optimality gap. However, unlike the strongly convex case, this measure does not provide a lower bound on the gradient norm. Thus, if the gradient becomes too small, due to the bias, the first-order oracle can always provide gradient estimates in the opposite direction to the true gradient, and hence make it impossible for the algorithm to progress any further. Thus, the stopping time for the convex case is in terms of both the optimality gap and the gradient norm. 
		
		Finally, we define the random variable $Z_k$ to measure progress towards optimality.
	\begin{defn}[Measure of Progress]
		\label{def:progress} For each $k \geq 0$, let $Z_k \geq 0$ be a random variable that measures the progress of the algorithm at step $k$.  The definition of $Z_k$ depends on the convexity of $\phi$. The corresponding definitions for each case are shown in the table below:
		\begin{table}[h!]
			\caption{ Definitions of $Z_k$.}
			\label{tbl:prog_upper}
			\centering 
			\begin{tabular}{lc}
				\toprule
				\textbf{Function} &
				\textbf{$\pmb{Z_k}$}  \\  \midrule
				\textbf{convex} &  $\frac{1}{\varepsilon_0} - \frac{1}{\phi(X_k) - \phi^*}$\\ \hdashline
				\textbf{non-convex} &  $\phi(X_k) - \phi^*$   \\ 
				\hdashline
				\textbf{strongly convex} & $\ln \left( \frac{\phi(X_k) - \phi^*}{\varepsilon} \right)$   \\
				\bottomrule
			\end{tabular}
		\end{table}
		This is analogous to a corresponding definition in \cite{berahas2019global}, except for a change in direction; for us, making progress means decreasing $Z_k$, whereas in \cite{berahas2019global}, making progress means increasing $Z_k$. 
	\end{defn}

	In the next section, using
	properties of processes $\{I_k\}$, $\{\Theta_k\}$ and  $\{Z_k\}$ we derive a high probability tail bound for  $T_\varepsilon$, and thus obtain a high probability bound on the  iteration complexity  for Algorithm \ref{alg:ls} when applied to non-convex, convex and strongly convex functions.

	\section{Analysis framework for the high probability bound}
	\label{sec:high_prob}	
In this section we present the main ingredients underlying the theoretical analysis. We first state general conditions on the stochastic process (Assumption \ref{ass:alg_behave}), from which we are able to derive a high probability tail bound on the iteration complexity.
They are listed as assumptions first, and in the following sections, we will show that they indeed hold for Algorithm \ref{alg:ls} when applied to non-convex, convex and strongly convex smooth functions $\phi$. 

%Then, a high probability concentration  bound on $T_\varepsilon$ under these conditions are shown. 

	\begin{restatable}
	[Properties of the stochastic process]{assumption}{algass}
		\label{ass:alg_behave}
		There exist a constant $\bar\alpha > 0$ and a non-decreasing function $h: \R\to\R$, which satisfies $h( \alpha)>0$ for any $\alpha>0$,
		{and a function $r: \R^2 \to \R$ which is non-decreasing and concave in its second argument,} and a constant $p \in ( \frac12 + \frac{{r(\epsilon_f^\prime, 2\epsilon_f^\prime)}}{h(\bar{\alpha})}, 1]$.
		such that the following hold for all $k<T_\varepsilon$:
		\begin{itemize}
		\item[(i)] $h(\bar{\alpha})>{\frac{r(\epsilon_f^\prime, 2\epsilon_f^\prime)}{p-\frac 1 2}}$. (The lower bound of potential progress an iteration with step size $\bar \alpha$ can make.)
		\item[(ii)] $\P(I_k = 1 \mid \sF_{k-1}) \geq p$ for all $k$. (Conditioning on the past, the next iteration is true with probability at least $p$.)
		      \item[(iii)]   If  $I_k\Theta_k=1$ then $Z_{k+1}\leq Z_k-h(A_k)+{r(\epsilon_f^\prime, 2\epsilon_f^\prime)}$. (True, successful iterations make progress.)
		       \item[(iv)] If $A_k  \leq  \bar \alpha$ and $I_k=1$  then $\Theta_k=1$. (Small and true iterations are also successful.) 
		       \item[(v)] $Z_{k+1}\leq Z_k+{r(\epsilon_f^\prime, {E}_k+{E}_k^+)}$ for all $k$. (The ``damage'' incurred at each iteration is bounded above.)	       
		 \end{itemize}
	\end{restatable}

	The following key lemma follows easily from Assumption \ref{ass:alg_behave} (ii) and the Azuma-Hoeffding inequality \rev{\cite{azuma1967weighted}} applied to the submartingale $\sum_{k=0}^{t-1} I_k-pt$.  
%	\ml{Do we need to mention the lemma in \cite{gratton2015direct}?}
	\begin{lemma}	\label{lem:azumaheoffding}
                For all $t\geq 1$, and any $\hat{p} \in [0, p)$, we have  
		$$\P\left(\sum_{k=0}^{t-1} I_k\leq \hat{p}t\right) \leq \exp\left( - \frac{(p-\hat{p})^2}{2p^2} t \right).$$
%		We will also need the following lemma from \cite{gratton2015direct}, which is essentially a one-sided concentration inequality for random variables which satisfy a submartingale-like condition. 
%			Suppose random variables $J_0, J_1, \ldots $ satisfy  $J_k \in \{0,1\}$ and
%			$$\text{$\P(J_0 = 1) \geq q$ and \,$\P(J_k = 1 \mid J_0, \ldots, J_{k-1}) \geq q$ for all $k \geq 1$.}$$ 
%			Then for all $n\geq 1$ and any $\hat{p} \in (\frac12, q)$, 
%			$$\P\left(\sum_{i=0}^{n-1} J_i\leq \hat p n\right) \leq \exp\left( - \frac{(q-\hat p)^2}{2q} n \right).$$
				
	\end{lemma}

	We now define another indicator variable that will be used in the analysis. 
\begin{defn}[Large step]
	For all integers $k \geq 0$, define the random variable $U_k$  as follows:
	\begin{align*}
		U_k &=
		\begin{cases}
			1, &\text{if $\min\{A_k, A_{k+1}\} \geq \bar{\alpha}$,} \\
			0, &\text{if $\max\{A_k, A_{k+1}\} \leq \bar{\alpha}$.}
		\end{cases}
	\end{align*}
	We will say that step $k$ is a \textbf{large step} if $U_k = 1$. Otherwise, step $k$ is a \textbf{small step}. Note that $U_k$ is adapted to the filtration $\mathcal{F}_k$, since $A_{k+1}$ is completely determined by $A_k$ and $\Theta_k$, which are both in $\mathcal{F}_k$.
\end{defn}
\rev{Without loss of generality (by possibly decreasing the value of $\bar{\alpha}$ by at most a factor of $\gamma$), we can assume that $\bar{\alpha} = \alpha_0 \gamma^m$ for some integer $m$.}  \revv{Then, using the dynamics of the process and considering all possible cases, it can be shown that every step is either a large step or a small step, and that the two possibilities are mutually exclusive.} 

%Thus during the first $t$ iterations, $\sum_{k=0}^{t-1} U_k \Theta_k$ is the total number of large  and  successful iterations, and  $\sum_{k=0}^{t-1}U_k ( 1 -\Theta_k)$ is the total
%number of large and unsuccessful iterations. 
Our analysis will rely on the following key observation: By Assumption \ref{ass:alg_behave}, if iteration $k$ has $U_k\Theta_kI_k=1$, then $Z_k$ gets reduced by at least 
$h(\bar \alpha)-{r(\epsilon_f^\prime, 2\epsilon_f^\prime)} > 0.$ We call such an iteration a {\bf {\em good}} iteration, because it makes progress towards optimality by at least a fixed amount. On the other hand, on any other iteration $k$, $Z_k$ can increase by at most ${r(\epsilon_f^\prime, E_k+E_k^+)}$. The idea of the analysis is to show that with high probability, the progress made by the good iterations dominates the damage caused by the other iterations. The crux of the proof is to show that with high probability, a large enough constant fraction of the iterations are good (up to another additive constant).

The engine of the analysis is a key lemma showing that if the stopping time has not been reached  and a large enough number of iterations are true, then there must be a large number of good iterations. 

	To prove the key lemma, we will first prove two additional lemmas. 
	The first lemma shows that the number of large and successful iterations is bounded below by the number of large and unsuccessful ones up to a constant. 	
	\begin{lemma}
		\label{lem:upper_updown}
		Let $d =  \max\left\{-\frac{\ln \alpha_0 - \ln \bar{\alpha}}{\ln \gamma},\; 0 \right\}$. For any positive integer $t$, we have
		$$
		\sum_{k=0}^{t-1} U_k \Theta_k \geq \sum_{k=0}^{t-1} U_k(1-\Theta_k) - d.
		$$
	\end{lemma}
	\begin{proof}
		The proof follows simply from the fact that  any unsuccessful step decreases the step size by a factor of $\gamma$, while any large successful step increases the step by a factor of $\gamma^{-1}$.  Since a large step at iteration $k$ has both $\alpha_k$ and $\alpha_{k+1}$ bounded from below by $\bar \alpha$, every time $\alpha_k$ gets decreased has to correspond to  a large step where it gets increased, except for at most $ \max\{-(\ln \alpha_0 - \ln \bar \alpha)/\ln \gamma, 0\}$ iterations, which is the number of unsuccessful steps it takes to decrease the step size from $\alpha_0$ to $\bar \alpha$. 
	\end{proof}
	Without loss of generality, one may assume $\alpha_0\geq \bar\alpha$, as $\alpha _0$ can be chosen to be large, \rev{and if $\alpha_0< \bar\alpha$, one can simply take $\bar\alpha= \alpha_0$ in the analysis}.
	\begin{corollary} 
		\label{cor:upper_success}
		From Lemma \ref{lem:upper_updown} we have
		$$\sum_{k=0}^{t-1} U_k\Theta_k \geq \frac12 \left(\sum_{k=0}^{t-1} U_k - d \right). $$
	\end{corollary}
	%	\begin{proof}[Proof of Corollary]
	%		Adding $\sum_{i=0}^{n-1}U_i\Theta_i$ to both sides of the inequality in Lemma \ref{lem:upper_updown}, we obtain
	%		$$2\sum_{i=0}^{n-1} U_i\Theta_i \geq \sum_{i=0}^{n-1} U_i(1-\Theta_i) + \sum_{i=0}^{n-1} U_i\Theta_i - d.$$
	%		Simplifying the right-hand side and dividing both sides by 2 gives the statement in the Corollary.
	%		
	%	\end{proof}
	
	The next Lemma is an analogue of Lemma \ref{lem:upper_updown} for the \emph{small} steps, it states that   the number of small true steps is upper-bounded by the number of small false steps.
	
	\begin{lemma}
		\label{lem:lower_updown}
		For any positive integer $t < T_\varepsilon$, we have:
		$$\sum_{k=0}^{t-1} (1-U_k)I_k \leq \sum_{k=0}^{t-1} (1-U_k)(1-I_k).$$
	\end{lemma}
	\begin{proof}
		We have
		\begin{align*}
			\sum_{k=0}^{t-1} (1-U_k)I_{k} \leq \sum_{k=0}^{t-1} (1-U_k)\Theta_{k} 
			\leq \sum_{k=0}^{t-1} (1-U_k)(1-\Theta_{k}) 
			\leq  \sum_{k=0}^{t-1} (1-U_k)(1-I_{k}).
		\end{align*}
		The first inequality follows from  Assumption \ref{ass:alg_behave} (iv), which implies that the number of small successful iterations is at least the number of small true iterations. The second inequality follows from the fact that the number of small steps where $\alpha_k$ is increased is bounded by the number of small steps where $\alpha_k$ is decreased. \revv{This is similar to the reasoning of Lemma 3.2, except there is no additive term because $\alpha_0 \geq \bar{\alpha}$.} The third inequality again uses Assumption \ref{ass:alg_behave} (iv), since any small unsuccessful has to be false.
	\end{proof}
	We are now ready to prove the key lemma.
		\begin{restatable}{lemma}{mainlemma}
		\label{lem:true_is_good}
		For any positive integer $t$ and any $\hat{p} \in (\frac12, 1]$, we have
		$$\P\left(\text{$T_\varepsilon > t$ and $\sum_{k=0}^{t-1}I_k \geq \hat{p}t$ and  $\sum_{k=0}^{t-1}U_k\Theta_kI_k < \left(\hat{p}-\frac12\right)t - \frac{d}{2}$ }\right) = 0,$$
		where $d = \max\left\{ - \frac{\ln \alpha_0 - \ln \bar{\alpha}}{\ln\gamma}, 0\right\}$. 
	\end{restatable}
	\begin{proof}
		It suffices to show that the two events $T_\varepsilon > t$ and $\sum_{k=0}^{t-1}I_k \geq \hat{p}t$ together imply the event $\sum_{k=0}^{t-1}U_k\Theta_kI_k \geq \left(\hat{p} - \frac12\right)t - \frac{d}{2}$. In the remainder of the proof, assume that $T_\varepsilon > t$ and $\sum_{k=0}^{t-1}I_k \geq \hat{p}t$.
		
		Among the first $t$ steps, let
		\begin{itemize}
			\item $L_t=\sum_{k=0}^{t-1} U_kI_{k}$ be the number of true large steps,
			\item $L_f=\sum_{k=0}^{t-1} U_k(1-I_{k})$  be the number of false large steps,
			\item $S_t=\sum_{k=0}^{t-1} (1-U_k)I_{k}$ be the number of true small steps,
			\item $S_f=\sum_{k=0}^{t-1} (1-U_k)(1-I_{k})$ be the number of false small steps,
			\item $L = L_t+L_f$ be the number of large steps,
			\item $S = S_t + S_f$ be the number of small steps.
		\end{itemize}
		Observe that $L+S = t$, because every step is either large or small. Moreover, since $\sum_{k=0}^{t-1} I_k \geq \hat{p}t$, this implies
		\begin{equation}
			L_f+S_f = \sum_{k=0}^{t-1} (1-I_k) \leq t-\hat{p}t.
			\label{eq:1}
		\end{equation}  
		Also, from Lemma \ref{lem:lower_updown} and the fact that $t < T_\varepsilon$, we know that 
		\begin{equation}S_t=\sum_{k=0}^{t-1} (1-U_k)I_{k} \leq  \sum_{k=0}^{t-1} (1-U_k)(1-I_{k})=S_f
			\label{eq:2}
		\end{equation}

		Now, recall from Corollary \ref{cor:upper_success} that the number of large, successful steps is  $ \sum_{k=0}^{t-1} U_k\Theta_k \geq \frac12 \left(L - d \right)$. 
		Also, note that
		$
		\sum_{k=0}^{t-1} U_k\Theta_k= \sum_{k=0}^{t-1} U_k\Theta_k I_k +  \sum_{k=0}^{t-1} U_k\Theta_k(1-I_k).
		$
		This implies that the number of large, successful, true steps is at least 
		\begin{align*}
			\sum_{k=0}^{t-1} U_k\Theta_k I_k
			&\geq \frac{L}{2} - \frac{d}{2} -  \sum_{k=0}^{t-1} U_k\Theta_k(1-I_k)
			\geq \frac{L}{2} - \frac{d}{2} -  \sum_{k=0}^{t-1} U_k(1-I_k)  \\[5pt]
			&= \frac{L}{2} - \frac{d}{2} -  L_f  
			= \frac{t-S_t-S_f}{2} - \frac{d}{2} - L_f \\[5pt]
			&\geq \frac{t-S_t-S_f}{2} - \frac{d}{2} - ((1-\hat{p
			})t-S_f) \qquad \text{(by \ref{eq:1})}\\[5pt]
			&= \frac{S_f-S_t}{2} + \left(\hat{p} - \frac12\right)t - \frac{d}{2} \\[5pt]
			&\geq \left(\hat{p} - \frac12\right)t - \frac{d}{2} \qquad \text{(by \ref{eq:2})} 
		\end{align*}
	\end{proof}

	\subsection{Bounded noise case}
	In \cite{CS17} and \cite{berahas2019global}, the \emph{expected} iteration complexity of the step search algorithm is bounded under the assumptions that $e(x)= 0$ and 
	$e(x)\leq \epsilon_f$ for all $x$, respectively. Let $\epsilon_f^\prime \geq \epsilon_f$. We now derive a \emph{high probability} tail bound on the iteration complexity under the assumption that $e(x)\leq \epsilon_f\leq \epsilon_f^\prime$ for all $x$. 
	We consider this case separately, because its analysis will inform the analysis of the general case of unbounded noise. In addition, in the unbounded noise case, we will need an additional assumption that the noise on different iterations is independent. Here, however, we allow for any type of noise, including adversarial.

Under Assumption \ref{ass:alg_behave} (iii) and (v) in the bounded noise setting, we have $Z_{k+1} \leq Z_k + {r(\epsilon_f^\prime, 2\epsilon_f^\prime)}$ in all iterations, and $Z_{k+1} \leq Z_k - h(\bar{\alpha}) + {r(\epsilon_f^\prime, 2\epsilon_f^\prime)}$ in good iterations. Putting this together with Lemma \ref{lem:true_is_good} and the other conditions in Assumption \ref{ass:alg_behave}, we obtain the following theorem.
	\begin{restatable}[Iteration complexity in the bounded noise setting]{theorem}{boundednoise}
		\label{thm:bounded_noise}
		Suppose Assumption \ref{ass:alg_behave} holds, and $E_{k}, E^+_{k} \leq \epsilon_f\leq \epsilon_f^\prime$ at every iteration.
               Then for any $\hat{p} \in (\frac12 + \frac{{r(\epsilon_f^\prime, 2\epsilon_f^\prime)}}{h(\bar{\alpha})}, p)$, and 
		$t\geq \frac{R }{\hat{p} - \frac12 - \frac{{r(\epsilon_f^\prime, 2\epsilon_f^\prime)}}{h(\bar{\alpha})}}$ we have
		$$\P\left(T_\varepsilon \leq t\right) \geq 1 - \exp\left(-\frac{(p-\hat{p})^2}{2p^2}t\right),$$
where $R = \frac{Z_0}{h(\bar{\alpha})}+\frac{d}{2}$ and $d = \max\left\{ - \frac{\ln \alpha_0 - \ln \bar{\alpha}}{\ln\gamma}, 0\right\}$.  
	\end{restatable}
		\begin{proof}
		In the bounded noise case, Assumption \ref{ass:alg_behave} tells us that as long as $k < T_\varepsilon$, we have $Z_{k+1} \leq Z_k - h(\bar{\alpha}) + {r(\epsilon_f^\prime, 2\epsilon_f^\prime)}$ if $U_kI_k\Theta_k = 1$,  and $Z_{k+1} \leq Z_k + {r(\epsilon_f^\prime, 2\epsilon_f^\prime)}$ if $U_kI_k\Theta_k = 0$. 
		
		The event $T_\varepsilon > t$ implies that $Z_{t}>0$ (\rev{since $Z_{t}= 0$ can only happen when the stopping time has been reached}, hence $T_\varepsilon \leq t$), this in turn implies  the event $\sum_{k=0}^{t-1}U_kI_k\Theta_k <(\hat p - \frac12)t - \frac{d}{2}$. To see this, assume that  $\sum_{k=0}^{t-1}U_kI_k\Theta_k \geq (\hat p - \frac12)t - \frac{d}{2}$, then 
		\begin{align*}
		Z_t &\leq Z_0 - \left[\left(\left(\hat{p} - \frac12\right)\cdot t - \frac{d}{2}\right) h(\bar{\alpha}) - t \cdot {r(\epsilon_f^\prime, 2\epsilon_f^\prime)}\right] \\
		&= Z_0 -  \left(\left(\hat{p}-\frac12\right)h(\bar\alpha)-{r(\epsilon_f^\prime, 2\epsilon_f^\prime)}\right)t + \frac{d}{2}\cdot h(\bar\alpha) \leq 0.
		\end{align*}
		The last inequality above used the assumptions that
		$\hat{p} \geq \frac12 + \frac{{r(\epsilon_f^\prime, 2\epsilon_f^\prime)}}{h(\bar{\alpha})}$ and 	$t \geq \frac{R}{\hat{p} - \frac12 - \frac{{r(\epsilon_f^\prime, 2\epsilon_f^\prime)}}{h(\bar{\alpha})}}$.
		
		Thus, we get
		\begin{align*}
			\P(T_\varepsilon > t) &= \P\left(T_\varepsilon > t, \; \sum_{k=0}^{t-1}U_k\Theta_kI_k < \left(\hat{p}-\frac12\right)t - \frac{d}{2}\right) \\[5pt]
			&= \P\left(T_\varepsilon > t, \; \sum_{k=0}^{t-1}U_k\Theta_kI_k < \left(\hat{p}-\frac12\right)t - \frac{d}{2},\; \sum_{k=0}^{t-1}I_k < \hat{p}{t}\right) \\[5pt]
			& \qquad + \P\left(T_\varepsilon > t, \; \sum_{k=0}^{t-1}U_k\Theta_kI_k < \left(\hat{p}-\frac12\right)t - \frac{d}{2},\; \sum_{k=0}^{t-1}I_k \geq \hat{p}{t}\right) \\[5pt]
			&\leq \P\left(\sum_{k=0}^{t-1}I_k < \hat{p} t \right) + 0  \leq\exp\left(-\frac{(p-\hat{p})^2}{2p^2}t \right).
		\end{align*}
		Here, the first equality is due to the fact that the event $T_\varepsilon > t$ implies the event $\sum_{k=0}^{t-1}U_k\Theta_kI_k < \left(\hat{p}-\frac12\right)t - \frac{d}{2}$. The first inequality uses  Lemma \ref{lem:true_is_good}, and the last inequality is by Lemma \ref{lem:azumaheoffding}.
	\end{proof}	
		
\subsection{General subexponential noise case}
\label{sec:subexp_noise}		
We now present a high probability bound for the iteration complexity with general  subexponential noise in the zeroth-order oracle. 
The result is very similar to that of Theorem \ref{thm:bounded_noise}. 
The main difference from the bounded noise analysis is that instead of bounding the ``damage" caused on a per-iteration basis, we bound the sum of all such damages over all iterations. 

\rev{
%\ml{removed the assumption. removed the reference to it (check everything is still okay.)}

We recall the definition of a subexponential random variable $X$ with parameter $(\nu,b)$ as follows:
\begin{equation*}
\mathbb E\left [\exp\{\lambda (X-\E[X])\}\right ]\leq \exp\left(\frac{\lambda^2\nu^2}{2}\right), \quad \forall \abs{\lambda}\in \left[0,\frac{1}{b}\right].
\end{equation*}

%\begin{assumption} \label{ass:indep_noise}
%		 The random variables ${e_0+e_0^+,e_1+e_1^+,\dots}$ are either deterministically bounded by  $\epsilon_f$ or are independent.
%		Either ${E_0,E_0^+,E_1,E_1^+,\dots}$ are all deterministically bounded by  $\epsilon_f$, or the random variables $\{E_0+E_0^+, E_1+E_1^+, \ldots\}$ are independent. 
%\end{assumption}
}

\begin{lemma}\label{lemma:concave_imply_subexp}
	For $r(\epsilon_f^\prime,\cdot)$ increasing continuous and concave in the second argument with $r(0,0)=0$, $r(\epsilon_f^\prime,E_k+E_k^+)$ is a subexponential random variable as a function of the random variable $E_k+E_k^+$.
\end{lemma}

\begin{proof}
	Another equivalent definition for $r(\epsilon_f^\prime,E_k+E_k^+)$ to be a subexponential random variable is if for some  $c>0$,
	$$\P\left(|r(\epsilon_f^\prime,E_k+E_k^+)|\geq t\right) \leq 2 e^{-ct} \text{ for all } t \geq 0.$$
	 \rev{(See for example Proposition 2.7.1 in \cite{vershynin2018high}.)} 
	Since $r(\epsilon_f^\prime,E_k+E_k^+)$ is always non-negative, we just need to show $\P\left(r(\epsilon_f^\prime,E_k+E_k^+)\geq t\right) \leq 2 e^{-ct} \text{ for all } t \geq 0.$
	
	$r(\epsilon_f^\prime,\cdot)$ is a single argument function from $\R_+\rightarrow\R_+$ that is increasing so it is invertible, with inverse function denoted $r^{-1}_{\epsilon_f^\prime}(\cdot)$ also increasing.  Let $[I_l,I_u)$ be the range of $r(\epsilon_f^\prime,\cdot)$, then the above inequality is automatically satisfied for all $t\geq I_u$.
	Hence, it remains to show that, for some $c>0$,
	$\P\left(r(\epsilon_f^\prime,E_k+E_k^+)\geq t\right) \leq 2 e^{-ct} \text{ for all } t \in [0,I_l),$ and 
	$\P\left(E_k+E_k^+\geq r_{\epsilon_f^\prime}^{-1}(t)\right) \leq 2 e^{-ct} \text{ for all }  t\in [I_l,I_u).$ 

	The first inequality can be assured by choosing \rev{$c\leq\frac {\ln(2)}{I_l}$} so that $e^{-cI_l}\geq \frac 1 2$. 
	For the second inequality, note that since $E_k+E_k^+$ is a subexponential random variable itself, for some $\hat c>0,$ it satisfies
	$\P\left(E_k+E_k^+\geq r_{\epsilon_f^\prime}^{-1}(t)\right) \leq 2 e^{-\hat cr_{\epsilon_f^\prime}^{-1}(t)} \text{ for all } t \in  [I_l,I_u).$ Moreover, by assumptions of  $r(\epsilon_f^\prime,\cdot)$, we have $r^{-1}_{\epsilon_f^\prime}(\cdot)$ is convex. 
	%\textbf{Approach 1}:
	Hence, by choosing $c>0$ sufficiently small, there exists \rev{$v > I_l$}, such that $2 e^{-\hat cr_{\epsilon_f^\prime}^{-1}(t)}\leq 2 e^{-ct}$ for all $t\geq v$.  By choosing $c'=\min\{c,\frac{\ln(2)}{v}\}$, we ensure $2 e^{-c't}\geq 1$ for all \rev{$t \in [I_\ell, v]$}. Thus, for such $c'>0$, we have that the above inequality holds for all $t\geq 0$.
\end{proof}
%	\ks{I am not sure if this is sufficienlty explicit in terms of constants.}
%	\textbf{Approach 2}:
%	One can also assume the derivative of $r_{\epsilon_f}(\cdot)$ at $0$ is positive, which implies the derivative of $r^{-1}_{\epsilon_f}(\cdot)$ at $0$ is also positive, say it is $c’>0$, then  $r^{-1}_{\epsilon_f}(x)$ dominates the linear function $c‘\cdot x-b$ for $x,b\geq 0$. This implies $2 e^{-\hat cr_{\epsilon_f}^{-1}(t)}\leq 2 e^{-\hat cc’t+\hat cb}$ for all $t\geq 0.$ Need to deal with this constant term $\hat cb$ as in approach 1..

Let ${(\nu_r,b_r)}$ be the parameters of the subexponential random variable $r(\epsilon_f^\prime, E_k+E_k^+)$.
% (A random variable $Y$ being subexponential with parameters $(\nu, b)$ means that
%$\E\left[e^{\lambda(Y - \E Y)}\right] \leq e^{\lambda^2\nu^2/2}$, for all $\abs{\lambda} < \frac1b$.) %{We want one-sided subexponential here.} 
The subexponential parameter will be derived for each specific $r(\epsilon_f^\prime,E_k+E_k^+)$ for each function class later. \rev{In the analysis of the iteration complexity, we will use the following Bernstein-like concentration inequality which gives an exponentially decaying bound on the probability that $\sum_k r(\epsilon_f', E_k + E_k^+)$ deviates from its mean.

%\begin{remark} 
%	We note that by the definition of the stochastic zeroth-order oracle, the random variables $E_k$ (and $E_k^+$) satisfy: 
%	$${\mathbb E}\left [\exp\{\lambda (E_k-\E[E_k])\}\mid \sF'_{k-1} \right ]\leq \exp\left(\frac{\lambda^2\nu^2}{2}\right), \quad \forall \lambda\in \left[0,\frac{1}{b}\right],$$ where $\sF'_{k-1}=\sigma (M_0, M_1, \ldots, M_{k-1}, \Xi'_k)$, which is the filtration generated by the randomness of the previous iterations and the randomness of the gradient estimate at the current iteration.
%
%\end{remark}

\begin{prop}
\label{prop:bernstein}
For all $t > 0$,
 $$\P\left(\frac1t\sum_{k=0}^{t-1} {r(\epsilon_f^\prime, E_k+E_k^+)}>  {\E\left[r(\epsilon_f^\prime, E_k+E_k^+)\right]} + s\right) \leq e^{-\min\{\frac{s^2t}{{2\nu_r^2}},\frac{st}{{2b_r}}\}}.$$
\end{prop}

\begin{proof}
\rev{We note that by the definition of the stochastic zeroth-order oracle, the random variables $E_k$ (and $E_k^+$) satisfy: 
	$${\mathbb E}\left [\exp\{\lambda (E_k-\E[E_k])\}\mid \sF'_{k-1} \right ]\leq \exp\left(\frac{\lambda^2\nu^2}{2}\right), \quad \forall \lambda\in \left[0,\frac{1}{b}\right],$$ where $\sF'_{k-1}=\sigma (M_0, M_1, \ldots, M_{k-1}, \Xi'_k)$, which is the filtration generated by the randomness of the previous iterations and the randomness of the gradient estimate at the current iteration.}

For clarity of notation, let $\mathcal{E}_k := r(\epsilon_f', E_k + E_k^+) - \E[r(\epsilon_f', E_k + E_k^+)]$ in this proof. By the preceding discussion and the properties of the zeroth-order oracle, $\mathcal{E}_k$ is $(\nu_r, b_r)$-subexponential conditioned on $\mathcal{F}'_{k-1}$. 
For any $\lambda \in (0, \frac{1}{b_r}]$, we have
\begin{align*}
\P\left(\frac{1}{t}\sum_{k=0}^{t-1}\sE_k > s\right)
= \P\left(\exp\left(\lambda\sum_{k=0}^{t-1}\sE_k\right) > e^{\lambda ts}\right) \leq e^{-\lambda ts}\, \E\left[ \exp\left(\lambda\sum_{k=0}^{t-1}\sE_k\right) \right]
\end{align*}
where the second inequality is by Markov's inequality. 

We claim that $\E\left[ \exp\left(\lambda\sum_{k=0}^{t-1}\sE_k\right) \right] \leq \exp\left(t\lambda^2\nu_r^2/2 \right)$. This can be shown by induction, since if $\E\left[ \exp\left(\lambda\sum_{k=0}^{m-1}\sE_k\right) \right] \leq \exp\left(m\lambda^2\nu_r^2/2 \right)$ for some $m$, then
\begin{align*}
&\E\left[ \exp\left(\lambda\sum_{k=0}^{m}\sE_k\right) \right]
= \E\left[ \exp\left(\lambda\sum_{k=0}^{m-1}\sE_k\right) \E\left[\exp(\lambda \sE_m) \mid \sF'_{m-1}\right]\right] \\
&\qquad \leq \exp(\lambda^2\nu_r^2/2) \,\E\left[ \exp\left(\lambda\sum_{k=0}^{m-1}\sE_k\right)\right]
\leq \exp\left((m+1)\lambda^2\nu_r^2/2 \right).
\end{align*}
Thus, for all $\lambda \in [0, \frac{1}{b_r}]$, we have
$$\P\left(\frac{1}{t}\sum_{k=0}^{t-1}\sE_k > s\right) \leq \exp\left(-\lambda t s + \frac{t\lambda^2\nu_r^2}{2} \right).$$
The result follows by setting $\lambda = \min\{\frac{1}{b_r}, \frac{s}{\nu_r^2}\}$, since:
\begin{itemize}
	\item If $\frac{s}{\nu_r^2} \leq \frac{1}{b_r}$, then $-\lambda ts + \frac{t\lambda^2\nu_r^2}{2} = -\frac{s^2t}{2\nu_r^2}$, and
	\item If $\frac{s}{\nu_r^2} > \frac{1}{b_r}$, then $-\lambda ts + \frac{t\lambda^2\nu_r^2}{2} = -\frac{ts}{b_r} + \frac{t\nu_r^2}{2b_r^2} \leq -\frac{ts}{2b_r}$.
\end{itemize}
Thus for this choice of $\lambda$, we have $\P\left(\frac{1}{t}\sum_{k=0}^{t-1}\sE_k > s\right) \leq \exp\left( -\min\{\frac{s^2t}{{2\nu_r^2}},\frac{st}{{2b_r}}\}\right)$ as claimed.
\end{proof}
}

We are now ready to introduce the main theorem for the iteration complexity.

\begin{restatable}[Iteration complexity in the subexponential noise setting]{theorem}{subexpnoise}
\label{thm:subexp_noise}
Suppose Assumptions \ref{epsilonf} and \ref{ass:alg_behave} hold. Then for any $s \geq 0$, $\hat{p} \in ( \frac12 + \frac{{r(\epsilon_f^\prime, 2\epsilon_f^\prime)}+s}{h(\bar{\alpha})}, p)$, and 
$t \geq \frac{R}{\hat{p} - \frac12 - \frac{{r(\epsilon_f^\prime, 2\epsilon_f^\prime)}+s}{h(\bar{\alpha})}}$, we have
$$\P\left(T_\varepsilon \leq t\right) \geq 1 - \exp\left(-\frac{(p-\hat{p})^2}{2p^2}t\right) - e^{-\min\{\frac{s^2t}{{2\nu_r^2}},\frac{st}{{2b_r}}\}},$$
where $R = \frac{Z_0}{h(\bar{\alpha})}+\frac{d}{2}$, $d = \max\left\{ - \frac{\ln \alpha_0 - \ln \bar{\alpha}}{\ln\gamma}, 0\right\}$ and ${\nu_r,b_r}$ are the parameters of the subexponential random variable $r(\epsilon_f^\prime, E_k+E_k^+)$. 
% with $s= r(\epsilon_f^\prime, 2\epsilon_f)-\E[r(\epsilon_f, e_k+e_k^+)].$
\end{restatable}
%\begin{remark} Note that if $e_k+e_k^+$ is deterministically bounded, then $\nu=b=0$ and Theorem \ref{thm:subexp_noise} reduces to Theorem \ref{thm:bounded_noise}, however, Theorem \ref{thm:bounded_noise} does not require independence of $e_k$, $e_k^+$ $k=0,1,\ldots$. 
%\end{remark}
\begin{proof}		
	By Assumption \ref{ass:alg_behave}, for all $k < T_\varepsilon$, we have $Z_{k+1} \leq Z_k - h(\bar{\alpha}) + {r(\epsilon_f^\prime, 2\epsilon_f^\prime)}$ if $U_kI_k\Theta_k = 1$,  and $Z_{k+1} \leq Z_k + {r(\epsilon_f^\prime, E_k+E_k^+)}$ if $U_kI_k\Theta_k = 0$. 
	%Let $\mu_k := \E[e_k + e_k^+]$. 
	% For simplicity assume that $e_k, e_k^+$ are all identically distributed. 
	% Let $\mu = \E[e_k]$.
	By the definition of the zeroth-order oracle  (\ref{eq:zero_order}) and Assumption \ref{epsilonf}, we know that $\E[E_k]$ and $\E[E_k^+]$ are bounded above by $\epsilon_f^\prime$ for all $k$. 
	By the law of total probability,
	\begin{align*}
		\P\left(T_\varepsilon > t\right) = 
		&\P\left(\underbrace{T_\varepsilon > t,\; \frac1t\sum_{k=0}^{t-1}{r(\epsilon_f^\prime, E_k+E_k^+)} \leq {r(\epsilon_f^\prime, 2\epsilon_f^\prime)} + s}_{A}\right) \\
		&\quad + 
		\P\left(\underbrace{T_\varepsilon > t,\; \frac1t\sum_{k=0}^{t-1}{r(\epsilon_f^\prime, E_k+E_k^+)} > {r(\epsilon_f^\prime, 2\epsilon_f^\prime)} + s}_{B}\right)
	\end{align*}
	
	First we bound $\P(B)$. {For each $k$, since $\E[E_k]$ and $\E[E_k^+]$ are both bounded above by $\epsilon_f^\prime$, we know that $$r(\epsilon_f^\prime, 2\epsilon_f^\prime) \geq r(\epsilon_f^\prime, \E[E_k+E_k^+]) \geq \E[r(\epsilon_f^\prime, E_k+E_k^+)].$$
	Here, the first inequality is because $r$ is non-decreasing in the second component, and the second inequality is by Jensen's inequality since $r$ is concave in its second component. Therefore, using \rev{\Cref{prop:bernstein}} gives for any $s \geq 0$:}
%	For each $k$, since $e_k$ and $e_k^+$ satisfy the one-sided sub-exponential bound \ref{eq:zero_order} with parameters $(\nu,b)$, it is not hard to show that $e_k + e_k^+$ satisfy \ref{eq:zero_order} with parameters $(2\nu, 2b)$. Moreover, since $e_k + e_k^+$ has mean bounded by $2\epsilon_f$,
%	applying (one-sided) Bernstein's inequality, gives for any $s \geq 0$:
	\[        \P(B) 
	\leq \P\left(\frac1t\sum_{k=0}^{t-1} {r(\epsilon_f^\prime, E_k+E_k^+)}>  {\E\left[r(\epsilon_f^\prime, E_k+E_k^+)\right]} + s\right) \leq e^{-\min\{\frac{s^2t}{{2\nu_r^2}},\frac{st}{{2b_r}}\}}.
	\]
	
	To bound $\P(A)$ we apply the law of total probability again,
	\begin{align*}
		\P(A) &= \P\left(\underbrace{T_\varepsilon > t,\; \frac1t\sum_{k=0}^{t-1}{r(\epsilon_f^\prime, E_k+E_k^+)} \leq {r(\epsilon_f^\prime, 2\epsilon_f^\prime)} + s, \; \sum_{k=0}^{t-1} \Theta_kI_kU_k < \left(\hat{p}-\frac12\right)t - \frac{d}{2}}_{A_1}\right) \\
		& \qquad +\P\left(\underbrace{T_\varepsilon > t,\; \frac1t\sum_{k=0}^{t-1}{r(\epsilon_f^\prime, E_k+E_k^+)} \leq {r(\epsilon_f^\prime, 2\epsilon_f^\prime)} + s, \; \sum_{k=0}^{t-1} \Theta_kI_kU_k \geq \left(\hat{p}-\frac12\right)t - \frac{d}{2}}_{A_2}\right). 
	\end{align*}
	Using the same logic as the first parts of the proof of   Theorem \ref{thm:bounded_noise}, one can show that $P(A_2)=0$ since 
	$T_\varepsilon > t$ and $\frac1t\sum_{k=0}^{t-1}{r(\epsilon_f^\prime, E_k+E_k^+)} \leq {r(\epsilon_f^\prime, 2\epsilon_f^\prime)} + s$ together imply that $\sum_{k=0}^{t-1}\Theta_kI_kU_k < \left(\hat{p}-\frac12\right)t - \frac{d}{2}$. Then, using the same argument as the second part of the proof of  Theorem \ref{thm:bounded_noise}, \revv{which uses \Cref{lem:azumaheoffding,lem:true_is_good}}, we have
	\begin{align*}
		\P(A_1)&\leq \P\left(T_\varepsilon > t,\; \sum_{k=0}^{t-1} \Theta_kI_kU_k < \left(\hat{p}-\frac12\right)t - \frac{d}{2}\right) \\
		&\leq \exp\left(-\frac{(p-\hat{p})^2}{2p^2}t\right).
	\end{align*}
	
	Combining $\P(A)$ and $\P(B)$, we conclude the proof. 			
\end{proof}
		
  \section{Iteration complexity of the step search algorithm} 
  \label{three_setting}
  
  \rev{In this section, we verify that Assumption \ref{ass:alg_behave} holds for the non-convex, strongly convex, and convex functions, and derive the expressions for the functions $h$ and $r$ in each case. Together with \Cref{thm:subexp_noise}, this gives the high probability iteration complexity bound for SASS for each of the three cases. The following fact about subexponential random variables will be used in the proofs in this section.
  \begin{prop}
 \label{prop:subexp}
  Let $X$ and $Y$ be (possibly dependent) subexponential random variables with parameters $(\nu_1, b_1)$ and $(\nu_2, b_2)$, respectively. Then, 
  \begin{enumerate}
  	\item For any scalars $a,b$, $aX + b$ is $(\abs{a}\nu_1, \abs{a}b_1)$-subexponential, and
	\item $X+Y$ is $(\sqrt{2(\nu_1^2 +\nu_2^2)}, \max\{2b_1, 2b_2\})$-subexponential.
  \end{enumerate}
  \end{prop}
  \begin{proof}
  For the first part, we bound the moment-generating function of $aX + b$ as follows:
  \begin{align*}
  	\E[e^{\lambda(aX + b - \E[aX+b])}] 
	= \E[e^{a\lambda(X -\E X)}]
	\leq e^{a^2\lambda^2\nu_1^2/2} \quad \text{for all $\abs{\lambda} \leq \frac{1}{\abs{a}b_1}$},
  \end{align*}
  where the last inequality is using the fact that $X$ is $(\nu_1, b_1)$-subexponential. This proves 1.
      
  For the second part, we bound the moment-generating function of $X + Y$ as follows:
  \begin{align*}
  	\E[e^{\lambda(X + Y - \E X - \E Y)}] 
	&\leq \left(\E[e^{2\lambda(X -\E X)}]  \E[e^{2\lambda(Y -\E Y)}]\right)^{\frac12} \quad \text{(by Cauchy-Schwarz)}\\
	&\leq \left(e^{4\lambda^2\nu_1^2/2} e^{4\lambda^2\nu_2^2/2}\right)^{\frac12} \quad \text{for all $\abs{\lambda} \leq \min\left\{\frac{1}{2b_1},\, \frac{1}{2b_2} \right\}$} \\
	&= e^{\lambda^2(\nu_1^2 + \nu_2^2)}.
  \end{align*}
  Therefore, $X + Y$ is $(\nu, b)$-subexponential for $\nu = \sqrt{2(\nu_1^2 + \nu_2^2)}$ and $b = \max\{2b_1, 2b_2\}$. 
  \end{proof}
  }
  
  \subsection{Non-convex case}
In the previous section, we presented high probability tail bounds on the iteration complexity under Assumption \ref{ass:alg_behave}. We now verify that Assumption \ref{ass:alg_behave} indeed holds for Algorithm \ref{alg:ls} when applied to smooth, possibly non-convex, functions. Together with the results in Section \ref{sec:high_prob}, this allows us to derive an explicit high-probability bound on the iteration complexity for non-convex functions.

As noted earlier, when either $\epsilon_f$ or $\epsilon_g$ are not zero, Algorithm \ref{alg:ls} does not converge to a stationary point, but converges to a neighborhood where $\|\nabla \phi(x)\|\leq \varepsilon$, with  $\varepsilon$ bounded from below in terms of $\epsilon_f^\prime (\geq \epsilon_f)$ and $\epsilon_g$.
The specific relationship is as follows.
%\begin{assumption}[Lower bound on the convergence criterion]
%	\begin{equation}\label{ass:eps_non_convex}
%		\varepsilon \geq \max\left\{\frac{\epsilon_g}{\eta}, \max \left\{{1+\kappa\alpha_{\mathrm{max}}}, \frac{1}{1-\eta}\right\}\cdot\sqrt{\frac{8\epsilon_f}{\theta} \cdot \max\left\{\frac{0.5L + \kappa}{1-\theta}, \frac{L(1-\eta)}{2(1-2\eta-\theta(1-\eta))} \right\}} \right\}.
%	\end{equation}
	\begin{restatable}[Lower bound on $\varepsilon$]{inequality}{epiasspt}
		\begin{equation*}\label{ass:eps_non_convex}
		\varepsilon > \max\left\{\frac{\epsilon_g}{\eta}, \max \left\{{1+\tau}, \frac{1}{1-\eta}\right\}\cdot\sqrt{\frac{4\epsilon_f^\prime}{\theta(p - \frac12)} \cdot \max\left\{\frac{0.5L + \kappa}{1-\theta}, \frac{L(1-\eta)}{2(1-2\eta-\theta(1-\eta))} \right\}} \right\},
		\end{equation*}
	for some $\eta\in (0,\frac{1-\theta}{2-\theta})$, and $p>\frac{1}{2}$. 
	\end{restatable} 
	%    Given $\epsilon_f$ and $\epsilon_g$ defined by the oracles, the following is a lower bound on $\varepsilon$:

%     $$\varepsilon \geq \max\left\{\frac{\epsilon_g}{\eta}, \sqrt{\frac{\epsilon_f}{\theta  \min\left\{\frac{1-\theta}{0.5L + \kappa}, \frac{2(1-2\eta-\theta(1-\eta))}{L(1-\eta)} \right\}}}\min \{\frac{1}{1+\kappa\alpha_{\mathrm{max}}}, 1-\eta\}\right\}.$$ 
   Here $\eta$ can be any value in the interval. $p = 1 - \delta$ when in the bounded noise setting, and $p = 1-\delta -\exp\left(-\min\{\frac{u^2}{\nu^2},\frac{u}{b}\}\right)$ otherwise, with $u = \inf_x \{\epsilon_f^\prime - \E[e(x)]\}$.  
%\end{assumption}

\begin{restatable}[Assumption \ref{ass:alg_behave} holds for Algorithm \ref{alg:ls}]{prop}{propthree}
\label{prop:ass_non-convex} 
If Inequality \ref{ass:eps_non_convex} and Assumptions \ref{ass:Lipschitz} and \ref{epsilonf} hold,  then Assumption \ref{ass:alg_behave} holds for Algorithm \ref{alg:ls} with the following $p$, $\bar \alpha$ and $h(\alpha)$:
	    		\begin{enumerate}
	    		    \item $p = 1 - \delta$ when the noise is bounded by $\epsilon_f^\prime$, and $p = 1-\delta -\exp\left(-\min\{\frac{u^2}{2\nu^2},\frac{u}{2b}\}\right)$ otherwise. Here $u = \inf_x \{\epsilon_f^\prime - \E[e(x)]\}$.
	    		    \item $\bar{\alpha} = \min\left\{\frac{1-\theta}{0.5L + \kappa}, \frac{2(1-2\eta-\theta(1-\eta))}{L(1-\eta)} \right\}$.
	    		    \item $h(\alpha) = \min\left\{\frac{\theta\varepsilon^2\alpha}{(1+\tau)^2}, \theta\alpha(1-\eta)^2\varepsilon^2\right\}$.
			    \item $r(\epsilon_f^\prime, E_k+E_k^+) = 2\epsilon_f^\prime + E_k + E_k^+$.
	    		\end{enumerate}
\end{restatable}

\begin{proof}
	% First, suppose we are in the noiseless setting, so that $r(\epsilon_f, e_{k,1},e_{k,2}) = r(0,0,0) = 0$ for all $k$. In this setting, the proposition was proved in Section 3.1 of \cite{CS17}. In the remainder of this proof, we bootstrap the results in the noiseless setting to the bounded noise and sub-exponential noise settings.
	We will show that each item in Assumption \ref{ass:alg_behave} holds. Throughout the proof, we use $f(x)$ to mean $f(x,\xi(x))$ for clarity.
	\begin{description}
		\item[(i)] We need to show that $h(\bar \alpha)>\frac{4\epsilon_f^\prime}{p-\frac 1 2}$. 
		Indeed, using 
		\[
		h(\bar \alpha) = \min\left\{\frac{1}{(1+\tau)^2},(1-\eta)^2\right\}\theta\varepsilon^2\bar \alpha,
		\]
		with
		\[
		\bar{\alpha} = \min\left\{\frac{1-\theta}{0.5L + \kappa}, \frac{2(1-2\eta-\theta(1-\eta))}{L(1-\eta)} \right\},
		\]
		and the Inequality \ref{ass:eps_non_convex} on $\varepsilon$, we conclude that (i) holds. 
		
		\item[(ii)] We denote $J_k :=\one\left\{\norm{{G}_k - \nabla \phi({X}_k)} \leq \max\{\epsilon_g, \min\{\tau,\kappa \sA_k\} \norm{{G}_k}\}\right\}$. 
		
		Clearly, we have 
		\begin{align*}
			\P\left(I_k = 0 \mid \sF_{k-1}\right)
			&= \P\left(J_k = 0 \;\;\text{or}\;\; E_{k}+E^+_{k} > 2\epsilon_f^\prime\mid \sF_{k-1}\right) \\
			&\leq \P\left(J_k = 0 \mid \sF_{k-1}\right) + \P\left( E_{k}+E^+_{k}> 2\epsilon_f^\prime\mid \sF_{k-1}\right).
		\end{align*}
%		\ks{Not sure if $\sF_{k-1}$ is correct here, since $e^+_{k}$ depends on $g$.} 
		The first term is bounded above by $\delta$, based on the use of the first-order oracle. The second term is always zero in the case when $\epsilon_f^\prime$ is a deterministic bound on the noise. Otherwise, let $\bar E_k = E_k\mid \sF_{k-1}$ and $\bar E_k^+ = E_k^+\mid \sF_{k-1}$, since $\bar E_k$ and $\bar E_k^+$ individually satisfy the one-sided subexponential bound \eqref{eq:zero_order} with parameters $\epsilon_f\leq \epsilon_f^\prime$ and $(\nu, b)$, for any $x$, 
		then by \Cref{prop:subexp} we have that $\bar E_k + \bar E_k^+$ satisfies \eqref{eq:zero_order} with parameters $2\epsilon_f\leq 2\epsilon_f^\prime$ and $(2\nu, 2b)$ (for any $x$). Hence, we can apply (one-sided) Bernstein's inequality, bounding the second term above by $e^{-\min\{\frac{u^2}{2\nu^2},\frac{u}{2b}\}}$. (Recall that $u = \inf_x \{\epsilon_f^\prime - \E[e(x)]\}$.) Thus, we have shown that
		$$\P(I_k = 1 \mid \sF_{k-1}) \geq p$$
		for all $k$, for the $p$ in the statement of this Proposition. 
		
		The fact $p \in (\frac12 + \frac{4\epsilon_f^\prime}{h(\bar{\alpha})}, 1]$ follows from the definitions of $h$ and $\bar{\alpha}$ in the statement of this Proposition, together with the Inequality \ref{ass:eps_non_convex} on $\varepsilon$. 
		
		\item[(iii)] Since iteration $k$ is true, we know that $\norm{G_k - \nabla \phi(X_k)} \leq \max\{\epsilon_g, \min\{\tau,\kappa \sA_k\}\norm{G_k}\}$. 
		We consider two cases:
		% \bjcomment{Can probably just cite \cite{berahas2019global} and \cite{CS17} papers here.}
		\begin{itemize}
			\item Suppose $\norm{{G}_k - \nabla \phi({X}_k)} \leq \min\{\tau,\kappa {A}_k\}\norm{{G}_k}$. By the triangle inequality, we get 
			$$\norm{{G}_k} \geq \frac{1}{1+\min\{\tau,\kappa \sA_k\}} \norm{\nabla\phi({X}_k)} \geq \frac{1}{1+\tau}\norm{\nabla \phi({X}_k)}.$$  Together with the fact that iteration $k$ is successful, we obtain $$f({X}_{k+1}) - f({X}_k) \leq - {A}_k\theta\norm{{G}_k}^2 + 2\epsilon_f^\prime \leq  - \frac{{A}_k\theta\norm{\nabla\phi({X}_k)}^2}{(1+\tau)^2} + 2\epsilon_f^\prime.$$
			\item Suppose $\norm{{G}_k - \nabla \phi({X}_k)} \leq \epsilon_g$. Since $k < T_\varepsilon$, we have $\norm{\nabla \phi({X}_k)} > \varepsilon\geq \frac{\epsilon_g}{\eta}$. This implies that $\norm{{G}_k - \nabla \phi({X}_k)} \leq \eta\norm{\nabla\phi({X}_k)}$. Rearranging this using the triangle inequality, we get that 
			$$\norm{{G}_k} \geq (1-\eta)\norm{\nabla\phi({X}_k)}.$$
			Putting this together with the fact that iteration $k$ is successful, we obtain
			$$f({X}_{k+1}) - f({X}_k) \leq - {A}_k\theta\norm{{G}_k}^2 + 2\epsilon_f^\prime \leq  - {A}_k\theta(1-\eta)^2\norm{\nabla\phi({X}_k)}^2 + 2\epsilon_f^\prime.$$
		\end{itemize}
		Combining the above two cases, we get that on any true, successful iteration with $k < T_\varepsilon$, the following inequality holds:
		$$f({X}_{k+1}) - f({X}_k) \leq -\min\left\{\frac{1}{(1+\tau)^2}, \, (1-\eta)^2\right\}{A}_k\theta\norm{\nabla\phi({X}_k)}^2 + 2\epsilon_f^\prime.$$
		By $k < T_\varepsilon$, we know $\norm{\nabla\phi({X}_k)} > \varepsilon$, so the above inequality implies $f({X}_{k+1}) - f({X}_k) \leq -h({A}_k) + 2\epsilon_f^\prime$. Finally,  because $E_{k}+E_{k}^+\leq 2\epsilon_f^\prime$ on true iterations, we get $\phi({X}_{k+1}) - \phi({X}_k) \leq -h({A}_k) + 4\epsilon_f^\prime$. Recall that $Z_k = \phi({X}_k) - \phi^*$, so $Z_{k+1} - Z_k = \phi({X}_{k+1}) - \phi({X}_k)$. This proves (iii). 
		
		\item[(iv)] We first show that if ${A}_k \leq \bar{\alpha}$ and $I_k = 1$, then 
		\begin{equation}
			\label{eq:ass1iinoiseless}
			\phi({X}_k - {A}_k{G}_k) \leq \phi({X}_k) - {A}_k\theta\norm{{G}_k}^2.
		\end{equation} 
		Since $I_k=1$,  $\norm{{G}_k - \nabla \phi({X}_k)} \leq \max\{\min\{\tau,\kappa {A}_k\}\norm{{G}_k}, \epsilon_g\}$. Just like in the proof of (iii), we consider two cases:
		\begin{itemize}
			\item Suppose $\norm{{G}_k - \nabla \phi({X}_k)} \leq \min\{\tau,\kappa {A}_k\}\norm{{G}_k}$. Then since ${A}_k \leq \bar{\alpha} \leq \frac{1-\theta}{0.5L + \kappa}$, by Assumption \ref{ass:Lipschitz} and Lemma 3.1 of \cite{CS17}, we have that \eqref{eq:ass1iinoiseless} holds. 
			\item Suppose $\norm{{G}_k - \nabla \phi({X}_k)} \leq \epsilon_g$. Since $k < T_\varepsilon$, we have $\norm{\nabla{\phi({X}_k)}} > \varepsilon \geq \frac{\epsilon_g}{\eta}$ by Inequality \ref{ass:eps_non_convex}. Therefore, $\norm{{G}_k - \nabla \phi({X}_k)} \leq \eta\norm{\nabla\phi({X}_k)}$. Combining this with the fact that ${A}_k \leq \bar{\alpha} \leq \frac{2(1-2\eta-\theta(1-\eta))}{L(1-\eta)}$, by Assumption \ref{ass:Lipschitz} and Lemma 4.3 of \cite{berahas2019global} (applied with $\epsilon_f^\prime = 0$), we have that \eqref{eq:ass1iinoiseless} holds. 
		\end{itemize}
		Now, recalling the definitions of $E_k$ and $E_k^+$ and using the fact that $E_{k} + E_{k}^+ \leq 2\epsilon_f^\prime $ (since $I_k = 1$), inequality (\ref{eq:ass1iinoiseless}) implies
		$$
		f({X}_k - {A}_k{G}_k) \leq f({X}_k) - {A}_k\theta\norm{{G}_k}^2 +E_{k}+E_{k}^+ \leq f({X}_k) - {A}_k\theta\norm{{G}_k}^2 +2\epsilon_f^\prime ,
		$$ 
		which proves (iv).

		\item[(v)] Note  that $Z_{k+1} = Z_k$ on any unsuccessful iteration, so the inequality holds trivially in that case. On the other hand, if iteration $k$ is successful, then by the modified Armijo condition, we have $$f({X}_{k+1}) - f({X}_k) \leq - {A}_k\theta\norm{{G}_k}^2 + 2\epsilon_f^\prime  \leq  2\epsilon_f^\prime .$$
		This implies that $\phi({X}_{k+1}) - \phi({X}_k) \leq 2\epsilon_f^\prime  + E_{k}+E_{k}^+$. Since $Z_{k+1} - Z_k = \phi({X}_{k+1}) - \phi({X}_k)$, (v) is proved. 
	\end{description}
\end{proof}

 In this setting, $r(\epsilon_f^\prime , E_k+E_k^+) = 2\epsilon_f^\prime  + E_k + E_k^+$. Therefore, \rev{by \Cref{prop:subexp}}, the subexponential parameters for $r(\epsilon_f^\prime , E_k+E_k^+)$ are $(\nu_r, b_r) = (2\nu, 2b)$.

Together with Theorem \ref{thm:subexp_noise}, we obtain the explicit complexity bound for Algorithm \ref{alg:ls}.
\begin{restatable}{theorem}{mainfull} 
	\label{thm:main_full}
	Suppose the Inequality \ref{ass:eps_non_convex} on $\varepsilon$ is satisfied for some $\eta\in (0,\frac{1-\theta}{2-\theta})$, and Assumption \ref{ass:Lipschitz} hold, then we have the following bound on the iteration complexity:
 For any $s \geq 0$, $\hat{p} \in ( \frac12 + \frac{4\epsilon_f^\prime +s}{C\varepsilon^2}, p)$, and 
$t \geq \frac{R}{\hat{p} - \frac12 - \frac{4\epsilon_f^\prime +s}{C\varepsilon^2}}$,
$$\P\left(T_\varepsilon \leq t\right) \geq 1 - \exp\left(-\frac{(p-\hat{p})^2}{2p^2}t\right) - \exp\left(-\min\left\{\frac{s^2t}{8\nu^2},\frac{st}{4b}\right\}\right).$$
%where
%\[
%%C=\min\left\{\frac{1}{(1+\kappa\alpha_{\mathrm{max}})^2},(1-\eta)^2,\frac{1-\theta}{0.5L + \kappa}, \frac{2(1-2\eta-\theta(1-\eta))}{L(1-\eta)}\right\}\theta,
%C=\min\left\{\frac{1}{(1+\kappa\alpha_{\mathrm{max}})^2},(1-\eta)^2,\bar \alpha \right\}\theta,
%\]
Here, $R = \frac{\phi(x_0)-\phi^*}{C\varepsilon^2} + \max\left\{-\frac{\ln \alpha_0 - \ln \bar \alpha}{2\ln \gamma}, 0\right\}$, $C=\min\left\{\frac{1}{(1+\tau)^2},(1-\eta)^2\right\}\bar \alpha \theta,$
with $p$ and $\bar \alpha$ as defined in Proposition \ref{prop:ass_non-convex}. 
\end{restatable}		
\begin{remark}
	\begin{enumerate}
\item
	Inequality \ref{ass:eps_non_convex} makes sure there exists some $\hat{p} \in ( \frac12 + \frac{4\epsilon_f^\prime +s}{C\varepsilon^2}, p)$ for some $s>0$.  
	The above theorem is for the general subexponential noise setting. In the bounded noise special case, the last term $\exp\left(-\min\left\{\frac{s^2t}{8\nu^2},\frac{st}{4b}\right\}\right)$ in the probability is not present.

%\ml{Maybe we don't need this following simplified constants version any more?}
%	The final result for non-convex functions is summarized below with simplified constants. 
%
%\begin{theorem}[Main convergence result with simplified constants] \label{thm:main_summary}
%	Suppose Assumptions \ref{ass:Lipschitz} and \ref{ass:indep_noise} hold, and (for simplicity) $\theta=\frac{1}{2}$ and $\kappa\geq \max\{L, 1\}$.
%	Then, for any
%	$$\varepsilon \geq 4\max\left\{\epsilon_g,(1+\tau)  \sqrt{(L+2\kappa){\epsilon_f}}\right\},$$  
%	we have the following bound on iteration complexity:
%	
%	For any $s \geq 0$, $p = 1-\delta -e^{-\min\{\frac{u^2}{2\nu^2},\frac{u}{2b}\}}$, $\hat{p} \in ( \frac12 + \frac{4\epsilon_f+s}{C\varepsilon^2}, p)$, and 
%	$t \geq \frac{R}{\hat{p} - \frac12 - \frac{4\epsilon_f+s}{C\varepsilon^2}}$,
%	%		 $C=\frac{1}{2}\min\left\{\frac{1}{(1+\kappa\alpha_{\mathrm{max}})^2},\frac{1}{L + 2\kappa}\right\}$,
%	$$\P\left(T_\varepsilon \leq t\right) \geq 1 - \exp\left(-\frac{(p-\hat{p})^2}{2p^2}t\right) - \exp\left(-\min\left\{\frac{s^2t}{8\nu^2},\frac{st}{4b}\right\}\right).$$
%	Here, $u = \inf_x \{\epsilon_f - \E[e(x)]\}$, $R = \frac{\phi(x_0)-\phi^*}{C\varepsilon^2}- \frac{\ln ((L+2\kappa)\alpha_0)}{\ln \gamma}$, and
%	$C=\frac{1}{2(L+2\kappa)(1+\tau)^2}$.
%\end{theorem}
\item 
	This theorem essentially shows that the iteration complexity of Algorithm \ref{alg:ls} is bounded by   a quantity on the order of
	\[
	\frac{1}{p - \frac12 - \frac{4\epsilon_f^\prime +s}{C\varepsilon^2}}\left (\frac{\phi(x_0)-\phi^*}{C\varepsilon^2} \right )
	\]
	with overwhelmingly high probability. If $p=1$ and $\epsilon_f^\prime =0$, the above quantity essentially recovers the iteration complexity of the deterministic algorithm. 
	\item \revv{The dependence of the iteration complexity on the Lipschitz constant $L$ is $O(L)$. This is the same as the dependence on $L$ in \cite{berahas2019global,CS17}, except the bounds in those papers are in expectation and ours is in high probability. In contrast, the iteration complexity on the Lipschitz constant $L$ in \cite{paquette2018stochastic} is $O(L^3)$, which has a worse dependence on $L$.}
		\end{enumerate}
\end{remark}	

\subsection{Strongly convex case}

We now apply the results to functions $\phi(x)$ that are strongly convex. 
We will verify that Assumption \ref{ass:alg_behave} holds in the strongly convex setting for function evaluations in all cases, namely: noiseless, with bounded noises or with i.i.d. subexponential noises. All results in this section hold for functions that satisfy	the Polyak-Lojasiewicz  inequality (or PL inequality) \cite{polyak1963gradient} as well, since we are only using PL inequality in the analysis.
% \subsubsection{Noiseless and Bounded Noise Settings}

\begin{assumption}\label{strongly_convex}
	$\phi$ is $\beta$-strongly convex, in other words
	$$\phi(x) \geq \phi(y) + \nabla\phi(y)^T(x-y) + \frac{\beta}{2}\norm{x-y}^2, \quad \text{for all $x, y \in \R^n$}.$$
\end{assumption}

Similar to the non-convex case, the presence of biased noise means that we can only hope to converge to a point $x$ in some neighbourhood of the optimal solution, where the radius of the neighbourhood is determined by the magnitude of the noise. Below, we quantify the relationship between $\varepsilon$, $\epsilon_g$ and $\epsilon_f^\prime$.

%	\begin{restatable}[Lower bound on $\varepsilon$ for strongly convex functions]{inequality}{epiassptstrongg}
%	\begin{equation*}\label{ass:eps_strongly_convex1}
%		\varepsilon > \max\left\{\frac{\epsilon_g^2}{2\beta\eta^2}, 
%		\frac{4\epsilon_f^\prime}{\left( 1-\min\left\{\frac{1}{(1+\tau)^2},(1-\eta)^2\right\} \theta \beta \cdot \min\left\{\frac{1-\theta}{0.5L + \kappa}, \frac{2(1-2\eta-\theta(1-\eta))}{L(1-\eta)} \right\} \right)^{\frac 1 2-p}-1}\right\},
%	\end{equation*}
%	for some $\eta\in (0,\frac{1-\theta}{2-\theta})$. 
%\end{restatable} 

%Note that the above inequality essentially implies $\varepsilon \geq 4\epsilon_f^\prime$: w.l.o.g. we may have $\kappa$, $\tau \geq 1$, similarly w.l.o.g. $\beta<1$, and we know $\theta<1$, $\frac{1-\theta}{0.5L + \kappa}\leq 1$ thus the numerator is $< (1-\frac 1 4\cdot 1\cdot 1)^{-\frac 1 2}-1=(\frac 3 4)^{-\frac 1 2}-1$, which implies $\varepsilon>4\epsilon_f^\prime$.
%
%Or we may use the following lower bound instead:
	\begin{restatable}[Lower bound on $\varepsilon$ for strongly convex functions]{inequality}{epiassptstrong}
	\begin{equation*}\label{ass:eps_strongly_convex}
		\varepsilon > \max\left\{\frac{\epsilon_g^2}{2\beta\eta^2}, 
		\frac{4\epsilon_f^\prime}{\left( 1-\min\left\{\frac{1}{(1+\tau)^2},(1-\eta)^2\right\} \theta \beta \cdot \min\left\{\frac{1-\theta}{0.5L + \kappa}, \frac{2(1-2\eta-\theta(1-\eta))}{L(1-\eta)} \right\} \right)^{\frac 1 2-p}-1},4\epsilon_f^\prime \right\},
	\end{equation*}
	for some $\eta\in (0,\frac{1-\theta}{2-\theta})$, and $p>\frac{1}{2}$. 
\end{restatable} 

We show Assumption \ref{ass:alg_behave} holds in the setting where $\phi(x)$ is strongly convex.
\begin{prop}[Assumption \ref{ass:alg_behave} holds for strongly-convex functions]
	\label{prop:ass_strongly_convex}
	If Inequality \ref{ass:eps_strongly_convex} and Assumptions \ref{ass:Lipschitz} and \ref{epsilonf} hold,  then Assumption \ref{ass:alg_behave} holds for Algorithm \ref{alg:ls} with the following $p$, $\bar \alpha$ and $h(\alpha)$:
	\begin{enumerate}
		\item Let $p = 1 - \delta$ (for noiseless and bounded noise), or $p = 1-\delta -\exp\left(-\min\{\frac{u^2}{2\nu^2},\frac{u}{2b}\}\right)$ otherwise. Here $u = \inf_x \{\epsilon_f^\prime - \E[e(x)]\}$. 
		\item $\bar{\alpha} = \min\left\{\frac{1-\theta}{0.5L + \kappa}, \frac{2(1-2\eta-\theta(1-\eta))}{L(1-\eta)} \right\}$,
		\item $h(\alpha) = \min\left\{ -\ln\left(1-\frac{\alpha\theta\beta}{(1+\tau)^2}\right), -\ln\left(1-\alpha\beta\theta(1-\eta)^2 \right)\right\}$,
		\item $r(\epsilon_f^\prime, E_{k}+E_{k}^+) = \ln\left(1 + \frac{2\epsilon_f^\prime + E_{k}+E_{k}^+}{\varepsilon} \right)$. 
	\end{enumerate}
%	Then the following hold for all $k<T_\varepsilon$:
%	\begin{itemize}
%		\item [(i)] $h(\bar \alpha)>\frac{r(\epsilon_f^\prime,2\epsilon_f^\prime)}{p-\frac 1 2},$
%		\item [(ii)] Iteration $k$ is true with probability at least $p$ conditioned on the outcome of all previous iterations $0,\ldots,k-1$, with some $p \in ( \frac12 + \frac{{r(\epsilon_f^\prime, 2\epsilon_f^\prime)}}{h(\bar{\alpha})}, 1]$.
%		\item[(iii)]   If iteration $k$ is true (i.e.  $I_k=1$) and successful, then $Z_{k+1}\leq Z_k-h(A_k)+r(\epsilon_f^\prime, 2\epsilon_f^\prime)$. 
%		\item[(iv)] If $A_k  \leq  \bar \alpha$ and iteration $k$ is true  then
%		iteration $k$ is also successful.
%		\item[(v)] $Z_{k+1}\leq Z_k+r(\epsilon_f^\prime, E_{k}+E_{k}^+)$ for all $k$. 
%	\end{itemize}
\end{prop}
\begin{proof}
Since $\phi$ is $\beta$-strongly convex, we know that it satisfy the PL inequality:
\begin{equation}
\label{eq:strongly_convex}
    \norm{\nabla\phi(x_k)}^2 \geq 2\beta\left(\phi(x_k) - \phi^*\right).
\end{equation}
\rev{(For example, see Theorem 2.1.10 of \cite{nesterov_2004}.)
Note this is weaker than strongly convexity, and having the PL inequality is sufficient for proving this proposition. Hence, all the analysis in this section automatically applies to functions that satisfy PL inequality.}
We will use this inequality in the proofs of (iii), (iv), and (v). 
	\begin{description}
		\item[(i)]	The proof of (i) relies on the lower bound for $\varepsilon$ in \Cref{ass:eps_strongly_convex}.  In more detail, we assumed that
	$$
	\varepsilon > \frac{4\epsilon_f^\prime}{\left( 1-\min\left\{\frac{1}{(1+\tau)^2},(1-\eta)^2\right\} \theta \beta \cdot \min\left\{\frac{1-\theta}{0.5L + \kappa}, \frac{2(1-2\eta-\theta(1-\eta))}{L(1-\eta)} \right\} \right)^{\frac 1 2-p}-1},
	$$
	which after plugging in the definition of $\bar{\alpha}$ becomes
	$$
	\varepsilon > \frac{4\epsilon_f^\prime}{\left( 1-\min\left\{\frac{1}{(1+\tau)^2},(1-\eta)^2\right\} \bar{\alpha}\theta \beta  \right)^{\frac 1 2-p}-1}.
	$$
	Rearranging the above inequality yields
	$$
	1 + \frac{4\epsilon_f^\prime}{\varepsilon}
	<
\left( 1-\min\left\{\frac{1}{(1+\tau)^2},(1-\eta)^2\right\} \bar{\alpha}\theta \beta  \right)^{\frac 1 2-p},
	$$
and taking logs on both sides gives
$r(\epsilon_f^\prime, 2\epsilon_f^\prime)< (p-\frac12)h(\bar{\alpha}).$
		
	\item[(ii)] The proof of (ii) is exactly the same as the corresponding proof for the non-convex setting. 
	
	\item[(iii)] We follow the same idea as the proof of (iii) in Proposition \ref{prop:ass_non-convex}, except using the new lower bound for $\varepsilon$.
	Since iteration $k$ is true, we know $\norm{{G}_k - \nabla \phi({X}_k)} \leq \max\{\epsilon_g, \min\{\tau,\kappa {A}_k\}\norm{{G}_k}\}$. There are two cases:
		% \bjcomment{Can probably just cite \cite{berahas2019global} and \cite{CS17} papers here.}
		\begin{itemize}
			\item Suppose $\norm{{G}_k - \nabla \phi({X}_k)} \leq \min\{\tau,\kappa {A}_k\}\norm{{G}_k}$. Then exactly as in the proof of (iii) in Proposition \ref{prop:ass_non-convex}, we get $$f({X}_{k+1}) - f({X}_k) \leq - {A}_k\theta\norm{{G}_k}^2 + 2\epsilon_f^\prime \leq  - \frac{{A}_k\theta\norm{\nabla\phi({X}_k)}^2}{(1+\tau)^2} + 2\epsilon_f^\prime.$$
			\item Suppose $\norm{{G}_k - \nabla \phi({X}_k)} \leq \epsilon_g$. Since $k < T_\varepsilon$, we have $\phi({X}_k) - \phi^\star> \varepsilon\geq \frac{\epsilon_g^2}{2\beta\eta^2}$. Combining this with the strong convexity condition (\ref{eq:strongly_convex}), we get that $\norm{\nabla\phi({X}_k)}^2 \geq \frac{\epsilon_g^2}{\eta^2}$ This implies that $\norm{{G}_k - \nabla \phi({X}_k)} \leq \eta\norm{\nabla\phi({X}_k)}$. Rearranging this using the triangle inequality, we get that 
			$$\norm{{G}_k} \geq (1-\eta)\norm{\nabla\phi({X}_k)}.$$
			Putting this together with the fact that iteration $k$ is successful, we obtain
			$$f({X}_{k+1}) - f({X}_k) \leq - {A}_k\theta\norm{{G}_k}^2 + 2\epsilon_f^\prime \leq  - {A}_k\theta(1-\eta)^2\norm{\nabla\phi({X}_k)}^2 + 2\epsilon_f^\prime.$$
		\end{itemize}
		Combining the above two cases, we get that on any true, successful iteration with $k < T_\varepsilon$, the following inequality holds:
	$$\phi({X}_{k+1}) - \phi({X}_k) \leq -\min\left\{\frac{1}{(1+\tau)^2}, \, (1-\eta)^2\right\}{A}_k\theta\norm{\nabla\phi({X}_k)}^2+2\epsilon_f^\prime+E_{k}+E_{k}^+.$$
	Since $\phi$ is $\beta$-strongly convex, we know that $\norm{\nabla\phi({X}_k)}^2 \geq 2\beta\left(\phi({X}_k) - \phi^*\right)$. Moreover, since $k < T_\varepsilon$, we know that $\phi({X}_k) - \phi^* > \varepsilon$. Plugging these into the above inequality, and letting $m := \min\left\{\frac{1}{(1+\tau)^2}, \, (1-\eta)^2\right\}{A}_k\theta\beta$ for clarity, we get
	$$\phi({X}_{k+1}) - \phi({X}_k) \leq - 2m\cdot(\phi({X}_k) - \phi^*)+\frac{2\epsilon_f^\prime+E_{k}+E_{k}^+}{\varepsilon}\cdot(\phi({X}_k) - \phi^*).$$
	Rearranging the above inequality, we get
	$$\frac{\phi({X}_{k+1})-\phi^*}{\phi({X}_k)-\phi^*}\leq \left(1 - 2m + \frac{2\epsilon_f^\prime+E_{k}+E_{k}^+}{\varepsilon}\right).$$
	Taking logs of both sides, and recalling our definition of $Z_k = \ln\left(\frac{\phi(X_k) - \phi^*}{\varepsilon}\right)$, we get
	$$Z_{k+1}-Z_k \leq \ln \left(1 - 2m + \frac{2\epsilon_f^\prime+E_{k}+E_{k}^+}{\varepsilon}\right).$$
	Since iteration $k$ is true, we know that $E_{k}+E_{k}^+ \leq 2\epsilon_f^\prime$, so that $\frac{2\epsilon_f^\prime + E_{k}+E_{k}^+}{\varepsilon} \leq \frac{4\epsilon_f^\prime}{\varepsilon} \leq 1$ by Inequality \ref{ass:eps_strongly_convex}. Therefore, we can rearrange the above inequality as follows:
	\begin{align*}
		Z_{k+1} - Z_k
		&\leq \ln \left(1 - m - m\cdot\frac{2\epsilon_f^\prime+E_{k}+E_{k}^+}{\varepsilon}+ \frac{2\epsilon_f^\prime+E_{k}+E_{k}^+}{\varepsilon}\right) \\[5pt]
		&= \ln \left[\left(1 - m\right)\left(1 + \frac{2\epsilon_f^\prime+E_{k}+E_{k}^+}{\varepsilon} \right)\right] \\[5pt]
		&= \ln \left(1 - m\right) + \ln\left(1 + \frac{2\epsilon_f^\prime+E_{k}+E_{k}^+}{\varepsilon} \right) \\[5pt]
		&\leq  -h({A}_k) + r(\epsilon_f^\prime, 2\epsilon_f^\prime).
	\end{align*}
	This proves (iii).
	
	\item[(iv)] This is proved similarly as in the non-convex setting, now using the new stopping criteria for the strongly convex setting and the new lower bound for $\varepsilon$. 
	
	We first show that if ${A}_k \leq \bar{\alpha}$ and $I_k = 1$, then 
		\begin{equation}
		\label{eq:ass1iv_stronglyconvex}
			\phi({X}_k - {A}_k{G}_k) \leq \phi({X}_k) - {A}_k\theta\norm{{G}_k}^2.
		\end{equation} 
		Since $I_k=1$,  $\norm{{G}_k - \nabla \phi({X}_k)} \leq \max\{\min\{\tau,\kappa {A}_k\}\norm{{G}_k}, \epsilon_g\}$. Just like in the proof of (iii), we consider two cases:
		\begin{itemize}
			\item Suppose $\norm{{G}_k - \nabla \phi({X}_k)} \leq \min\{\tau,\kappa {A}_k\}\norm{{G}_k}$. Then since ${A}_k \leq \bar{\alpha} \leq \frac{1-\theta}{0.5L + \kappa}$, by Assumption \ref{ass:Lipschitz} and Lemma 3.1 of \cite{CS17}, we have \eqref{eq:ass1iv_stronglyconvex} holds. 
			\item Suppose $\norm{{G}_k - \nabla \phi({X}_k)} \leq \epsilon_g$. Since $k < T_\varepsilon$, we have $\phi({X}_k) - \phi^\star> \varepsilon\geq \frac{\epsilon_g^2}{2\beta\eta^2}$. Combining this with the strong convexity condition (\ref{eq:strongly_convex}), we get that $\norm{\nabla\phi({X}_k)}^2 \geq \frac{\epsilon_g^2}{\eta^2}$ This implies that $\norm{{G}_k - \nabla \phi({X}_k)} \leq \eta\norm{\nabla\phi({X}_k)}$. Combining this with the fact that ${A}_k \leq \bar{\alpha} \leq \frac{2(1-2\eta-\theta(1-\eta))}{L(1-\eta)}$, by Assumption \ref{ass:Lipschitz} and Lemma 4.3 of \cite{berahas2019global} (applied with $\epsilon_f^\prime = 0$), we have that \eqref{eq:ass1iv_stronglyconvex} holds. 
		\end{itemize}
		Now, recalling the definitions of $E_k$ and $E_k^+$ and using the fact that $E_{k} + E_{k}^+ \leq 2\epsilon_f^\prime$ (since $I_k = 1$), inequality (\ref{eq:ass1iinoiseless}) implies
		$$
		f({X}_k - {A}_k{G}_k) \leq f({X}_k) - {A}_k\theta\norm{{G}_k}^2 +E_{k}+E_{k}^+ \leq f({X}_k) - {A}_k\theta\norm{{G}_k}^2 +2\epsilon_f^\prime,
		$$ 
		which proves (iv). 
	
	\item[(v)] Finally, we turn to proving (v). The proof of (v) follows the same steps as the proof of (iii), but is simpler. If $k$ is an unsuccessful step, then $Z_{k+1} = Z_k$, so the inequality in (v) is clearly satisfied. If $k$ is successful, then
	$$\phi({X}_{k+1}) - \phi({X}_k) \leq 2\epsilon_f^\prime + E_{k} + E_{k}^+ \leq \frac{2\epsilon_f^\prime + E_{k}+E_{k}^+}{\varepsilon}\cdot\left(\phi({X}_k) - \phi^*\right).$$
	Here, the first inequality is by the sufficient decrease condition, and the second inequality is because $k < T_\varepsilon$. Rearranging, we get
	$$
	\frac{\phi({X}_{k+1}) - \phi^*}{\phi({X}_k) - \phi^*} \leq \left(1 + \frac{2\epsilon_f^\prime + E_{k} + E_{k}^+}{\varepsilon}\right).
	$$
	Taking logs on both sides gives $Z_{k+1} - Z_k \leq \ln\left(1 + \frac{2\epsilon_f^\prime + E_{k} + E_{k}^+}{\varepsilon}\right)$, as desired. 
	\end{description}	
\end{proof}
	
The proposition below gives the explicit subexponential parameters of $r(\epsilon_f^\prime, E_{k}+E_{k}^+)$. Recall that $r(\epsilon_f^\prime, E_{k}+E_{k}^+) = \ln\left(1 + \frac{2\epsilon_f^\prime + E_{k}+E_{k}^+}{\varepsilon} \right)$ in this strongly convex setting, and \rev{by \Cref{prop:subexp}, $1 + \frac{2\epsilon_f^\prime + E_{k}+E_{k}^+}{\varepsilon}$, is $(\frac{2\nu}{\eps}, \frac{2b}{\eps})$-subexponential.} 
%for a $(\nu,b)$ sub-exponential random variable  $W$, $a_1W+a_2$ for $a_1>0$ and $a_2 \in \R$ is a $(a_1\nu,a_1b)$ sub-exponential random variable. 
Together with the following proposition, the parameters of $r$ for the strongly convex setting are
	%if $x\geq 1$ is sub-exponential with parameter $c$ in the definition of probability tail bound, i.e., 
	%$\P(x\geq t)\leq 2e^{-ct}$ for all $t\geq 0$, then $\log(x)$ is sub-Gaussian with parameter $c$ in the definition of probability tail bound:
	%$\P(\log(x)\geq t)=\P(x\geq e^t)\leq 2e^{-ce^t}\leq 2e^{-ct^2},$ for all $t\geq 0$.
	$${\nu_r = b_r = 4e^2\max\left\{\frac{2\nu}{\varepsilon}, \frac{2b}{\varepsilon}\right\} + 4e\left(1+\frac{4\epsilon_f^\prime}{\varepsilon}\right)}.$$

		\begin{prop}
			\label{prop:log_subexp}
			Let $W \geq 1$ be a $(\nu, b)$-subexponential random variable. Then $\ln(W)$ is $(\nu', b')$-subexponential with $\nu' = b' = 4e^2\max\{\nu, b\} + 4e \E W$. 
		\end{prop}
		\begin{proof}
			For a random variable $W$, let $\norm{W}_{\psi_1} := \sup_{k\geq 1} \frac1k \left(\E\abs{W}^k\right)^{\frac1k}$. Using the fact that $\norm{\,\cdot\,}_\psi$ is a norm, we have
			\begin{align*}
				\norm{\ln(W) - \E\ln(W)}_{\psi_1}
				&\leq \norm{\ln(W)}_{\psi_1} + \norm{\E\ln(W)}_{\psi_1} \quad \text{(triangle inequality)}\\
				&= \norm{\ln(W)}_{\psi_1} + \E\ln(W) \quad \text{(definition of $\norm{\,\cdot\,}_\psi$)}\\
				&\leq \norm{W}_{\psi_1} + \E\ln(W) \quad \text{($0 \leq \ln(W) \leq W$, since $W \geq 1$)}\\
				&\leq \norm{W - \E W}_{\psi_1} + \E W + \E\ln(W) \quad \text{(triangle inequality).}
			\end{align*}
			By Proposition 2.7.1 (e) $\rightarrow$ (b) of \cite{vershynin2018high}, we have $\norm{W- \E W}_{\psi_1} \leq 2e\max\{\nu, b\}$. Thus, 
			$$\norm{\ln(W) - \E\ln(W)}_{\psi_1} \leq 2e \max\{\nu, b\} + \E W + \E\ln(W).$$
			Applying Proposition 2.7.1 (b) $\rightarrow$ (e) of \cite{vershynin2018high}, we conclude that $\ln(W)$ is $(\nu', b')$-subexponential where $\nu' = b' = 4e^2\max\{\nu, b\} + 2e \left(\E W + \E \ln(W)\right)$. 
		\end{proof}

	Putting things together, we obtain the following theorem, which bounds the iteration complexity in the strongly convex case.
	
\begin{theorem}
	{
		Suppose Inequality \ref{ass:eps_strongly_convex} on $\varepsilon$ holds for some $\eta\in (0,\frac{1-\theta }{2-\theta})$, and Assumptions \ref{ass:Lipschitz} and \ref{strongly_convex} hold. Then we have:
		For any $s \geq 0$, $\hat{p} \in ( \frac12 + \frac{\ln\left(1 + \frac{4\epsilon_f^\prime}{\varepsilon} \right)+s}{C}, p)$, and 
		$t \geq \frac{R}{\hat{p} - \frac12 - \frac{\ln\left(1 + \frac{4\epsilon_f^\prime}{\varepsilon} \right)+s}{C}}$, with 
		$C = -\max\left\{\ln\left(1-\frac{\bar\alpha\theta\beta}{(1+\tau)^2}\right),\ln\left(1-\bar\alpha\beta\theta(1-\eta)^2 \right)\right\}$,
		\begin{align*}
		&\P\left(T_\varepsilon \leq t\right) \geq 1 - \exp\left(-\frac{(p-\hat{p})^2}{2p^2}t\right) \\
		&- \exp\left(-\min\left\{\frac{s^2t}{2\left(4e^2\max\left\{\frac{2\nu}{\varepsilon}, \frac{2b}{\varepsilon}\right\} + 4e\left(1+\frac{4\epsilon_f^\prime}{\varepsilon}\right)\right)^2},\frac{st}{2\left(4e^2\max\left\{\frac{2\nu}{\varepsilon}, \frac{2b}{\varepsilon}\right\} + 4e\left(1+\frac{4\epsilon_f^\prime}{\varepsilon}\right)\right)}\right\}\right).
		\end{align*}
		Here, 
		$R = \frac{1}{C}{\ln\left(\frac{\phi(x_0) - \phi^*}{\varepsilon}\right)} + \max\left\{-\frac{\ln \alpha_0 - \ln \bar \alpha}{2\ln \gamma}, 0\right\}$, 
		with $p$ and $\bar \alpha$ as defined in Proposition \ref{prop:ass_strongly_convex}.
	}
	
\end{theorem}

	\subsection{Convex case}
	We now apply the results to functions $\phi(x)$ that are  convex. 
	We will verify that Assumption \ref{ass:alg_behave} also holds in the convex setting.
	
	\begin{assumption}\label{ass:convex}
		$\phi$ is convex, and there exists a constant $D > 0$ such that
		$$\left\|x-x^{\star}\right\| \leq D \quad \text { for all } x \in \mathcal{U},$$
		where $x^{\star}$ is some global minimizer of $\phi$, and the set $\mathcal{U}$ contains all
		iteration realizations.
	\end{assumption}
	
	For the convex case, we define
	the stopping time $T_{\varepsilon}$ to be the first time either $\phi({X}_k)-\phi^*\leq \varepsilon_0$ or $\norm{\nabla\phi({X}_k)}\leq \varepsilon_1$. 
	
	The presence of biased noise means that one can only hope to converge to a point $x$ in some neighborhood of the optimal solution, where the radius of the neighborhood is determined by the magnitude of the noise. 	
	\rev{
%	Moreover, since for general convex functions (unlike strongly convex functions), the norm of the gradient is not lower bounded by the magnitude of the optimality gap and because the gradient estimates are also biased,  the convergence neighborhood is defined in terms of both the gradient norm and the optimality gap value. 
%	Moreover, note that for general convex functions, we measure the progress using the optimality gap. However, unlike the strongly convex case, this measure does not provide a lower bound on the gradient norm. Thus, if the gradient becomes too small, due to the bias, the first-order oracle can always provide gradient estimates in the opposite direction to the true gradient, and hence make it impossible for the algorithm to progress any further. Thus, the stopping time for the convex case is in terms of both the optimality gap and the gradient norm. 
}

Below, we quantify the neighborhood of convergence.
	
	\begin{restatable}[Lower bound on $\varepsilon_0$ and $\varepsilon_1$ for  convex functions]{inequality}{epiassconvex}
		\begin{equation*}\label{ass:eps_convex}
			\varepsilon_0 > \max\left\{\sqrt{\frac{16D^2\epsilon_f^\prime}{\theta(p-\frac1 2)\min\left\{(1-\eta)^2,
			\frac{1}{(1+\tau)^2}\right\}\min\left\{\frac{1-\theta}{0.5L + \kappa}, \frac{2(1-2\eta-\theta(1-\eta))}{L(1-\eta)} \right\}}}, 
			4\epsilon_f^\prime\right\}, 
			\varepsilon_1\geq  \frac{\epsilon_g}{\eta}
		\end{equation*}
		for some $\eta\in (0,\frac{1-\theta}{2-\theta})$, and $p>\frac{1}{2}$. 
	\end{restatable}

	We restate Assumption \ref{ass:alg_behave} below in the form of a Proposition for the setting where $\phi(x)$ is convex:
	\begin{prop}[Assumption \ref{ass:alg_behave} holds for convex functions]
		\label{prop:ass_convex}
		Let
		\begin{enumerate}
			\item Let $p = 1 - \delta$ (for noiseless and bounded noise), or $p = 1-\delta -\exp\left(-\min\{\frac{u^2}{2\nu^2},\frac{u}{2b}\}\right)$ otherwise. Here $u = \inf_x \{\epsilon_f^\prime - \E[e(x)]\}$. 
			\item $\bar{\alpha} = \min\left\{\frac{1-\theta}{0.5L + \kappa}, \frac{2(1-2\eta-\theta(1-\eta))}{L(1-\eta)} \right\}$,
			\item $h(\alpha) = \frac{\alpha\theta}{4D^2}\min\left\{(1-\eta)^2,
			\frac{1}{(1+\tau)^2}\right\}$, 
			\item $r(\epsilon_f^\prime, E_{k}+E_{k}^+) = \frac{2\epsilon_f^\prime + E_{k}+E_{k}^+}{\varepsilon_0^2} $. 
		\end{enumerate}
		Then the following hold for all $k<T_\varepsilon$:
		\begin{itemize}
			\item [(i)] $h(\bar \alpha)>\frac{r(\epsilon_f^\prime,2\epsilon_f^\prime)}{p-\frac 1 2},$
			\item [(ii)] Iteration $k$ is true with probability at least $p$ conditioned on the outcome of all previous iterations $0,\ldots,k-1$, with some $p \in ( \frac12 + \frac{{r(\epsilon_f^\prime, 2\epsilon_f^\prime)}}{h(\bar{\alpha})}, 1]$.
			\item[(iii)]   If iteration $k$ is true (i.e.  $I_k=1$) and successful, then $Z_{k+1}\leq Z_k-h(A_k)+r(\epsilon_f^\prime, 2\epsilon_f^\prime)$. 
			\item[(iv)] If $A_k  \leq  \bar \alpha$ and iteration $k$ is true  then
			iteration $k$ is also successful.
			\item[(v)] $Z_{k+1}\leq Z_k+r(\epsilon_f^\prime, E_{k}+ E_{k}^+)$ for all $k$.
		\end{itemize}
	\end{prop}

	\begin{proof}
		% First, suppose we are in the noiseless setting, so that $r(\epsilon_f, e_{k,1},e_{k,2}) = r(0,0,0) = 0$ for all $k$. In this setting, the proposition was proved in Section 3.1 of \cite{CS17}. In the remainder of this proof, we bootstrap the results in the noiseless setting to the bounded noise and sub-exponential noise settings.
		(i) can be easily verified using the lower bound of $\varepsilon_0$ and $\varepsilon_1$, the definition of $h(\alpha), \bar \alpha$ and $r(\epsilon_f^\prime,2\epsilon_f^\prime)$.
		
		The proofs of (ii) and (iv) are exactly the same as the corresponding proofs for the non-convex setting with the new stopping time and the new definitions for functions $h$ and $r$ for the convex case. 
		
		Next, we prove (iii). 
		By assumption \ref{ass:convex} and Cauchy-Schwarz, $$\phi(x^*)-\phi({X}_k)\geq \nabla\phi({X}_k)^T(x^*-{X}_k)\geq -D\norm{\nabla\phi({X}_k)},$$ for any ${X}_k$, $x^*$ is a minimizer of $\phi$.
		
		Since iteration $k$ is true, we know that $\norm{{G}_k - \nabla \phi({X}_k)} \leq \max\{\epsilon_g, \min\{\tau,\kappa {A}_k\}\norm{{G}_k}\}$. 
		We again consider two cases:
		% \bjcomment{Can probably just cite \cite{berahas2019global} and \cite{CS17} papers here.}
		\begin{itemize}
			\item Suppose $\norm{{G}_k - \nabla \phi({X}_k)} \leq \min\{\tau,\kappa {A}_k\}\norm{{G}_k}$. By the triangle inequality, we get 
			$$\norm{{G}_k} \geq \frac{1}{1+\min\{\tau,\kappa {A}_k\}} \norm{\nabla\phi({X}_k)} \geq \frac{1}{1+\tau}\norm{\nabla \phi({X}_k)}.$$  
%			Together with the fact that iteration $k$ is successful, we obtain $$f({X}_{k+1}) - f({X}_k) \leq - {A}_k\theta\norm{{G}_k}^2 + 2\epsilon_f \leq  - \frac{{A}_k\theta\norm{\nabla\phi({X}_k)}^2}{(1+\kappa\alpha_{\mathrm{max}})^2} + 2\epsilon_f.$$
			\item Suppose $\norm{{G}_k - \nabla \phi({X}_k)} \leq \epsilon_g$. Since $k < T_\varepsilon$, we have $\norm{\nabla \phi({X}_k)} > \varepsilon_1\geq \frac{\epsilon_g}{\eta}$. This implies that $\norm{{G}_k - \nabla \phi({X}_k)} \leq \eta\norm{\nabla\phi({X}_k)}$. Rearranging this using the triangle inequality, we get that 
			$$\norm{{G}_k} \geq (1-\eta)\norm{\nabla\phi({X}_k)}.$$
%			Putting this together with the fact that iteration $k$ is successful, we obtain
%			$$f({X}_{k+1}) - f({X}_k) \leq - {A}_k\theta\norm{{G}_k}^2 + 2\epsilon_f \leq  - {A}_k\theta(1-\eta)^2\norm{\nabla\phi({X}_k)}^2 + 2\epsilon_f.$$
		\end{itemize}
		Hence, $$\norm{{G}_k} \geq \min\left\{(1-\eta),
		\frac{1}{1+\tau}\right\}		
		 \norm{\nabla\phi({X}_k)}.$$
%		 Denote $m'=\min\left\{1-\eta,
%		 \frac{1}{1+\tau}\right\}$. 
		 Following the exact argument of Lemma 4.8 in \cite{berahas2019global}, we have 
		 $$Z_{k+1}\leq Z_k-\frac{{A}_k\theta}{4D^2}\min\left\{(1-\eta)^2,	\frac{1}{(1+\tau)^2}\right\}+\frac{4\epsilon_f^\prime}{\varepsilon_0^2}.$$
	%		Since $k < T_\varepsilon$, we know that $\phi({X}_k) - \phi^* > \varepsilon$. .....
		This proves (iii).
		
		Finally, we turn to proving (v). The proof of (v) follows the same steps as the proof of (iii), but is simpler. If $k$ is an unsuccessful step, then $Z_{k+1} = Z_k$, so the inequality in (v) is clearly satisfied.
		
		If $k$ is successful, using a similar argument as in Lemma 4.9 in \cite{berahas2019global} \rev{(with $\frac{4\epsilon_f}{\eps^2}$ being replaced by $r(\epsilon_f', E_k + E_k^+) =\frac{2\epsilon_f^\prime + E_{k}+E_{k}^+}{\varepsilon_0^2}$)}, 
		we have
		$$Z_{k+1}\leq Z_k+\frac{2\epsilon_f^\prime + E_{k}+E_{k}^+}{\varepsilon_0^2}.$$ 
		This completes the proof.
	\end{proof}
		
In this setting, $r(\epsilon_f^\prime, E_{k}+E_{k}^+) = \frac{2\epsilon_f^\prime + E_{k}+E_{k}^+}{\varepsilon_0^2}$.
%If $x$ is a $(\nu,b)$ sub-exponential random variable, then $a_1x+a_2$ for $a_1>0$ and $a_2 \in \R$ is a $(a_1\nu,a_1b)$ sub-exponential random variable. 
By \rev{\Cref{prop:subexp}}, the subexponential parameters of $r(\epsilon_f^\prime, E_k+E_k^+)$ are $(\nu_r, b_r) = (2\nu/\varepsilon_0^2, 2b/\varepsilon_0^2)$.

Together with Theorem \ref{thm:subexp_noise}, we obtain the explicit complexity bound in the convex setting.
		
		\begin{theorem}
				Suppose the Inequality \ref{ass:eps_convex} on $\varepsilon_0, \varepsilon_1$ is satisfied for some $\eta\in (0,\frac{1-\theta}{2-\theta})$, and Assumptions \ref{ass:Lipschitz} and \ref{ass:convex} hold, then we have the following bound on the iteration complexity:
 For any $s \geq 0$, $\hat{p} \in ( \frac12 + \frac{4\epsilon_f^\prime/\varepsilon_0^2+s}{C}, p)$, and 
$t \geq \frac{R}{\hat{p} - \frac12 - \frac{4\epsilon_f^\prime/\varepsilon_0^2+s}{C}}$,
$$\P\left(T_\varepsilon \leq t\right) \geq 1 - \exp\left(-\frac{(p-\hat{p})^2}{2p^2}t\right) - \exp\left(-\min\left\{\frac{s^2t\varepsilon_0^4}{8\nu^2},\frac{st\varepsilon_0^2}{4b}\right\}\right).$$
Here, $R = \frac{1}{C} \cdot \left(\frac{1}{\varepsilon_0} - \frac{1}{\phi(x_0) - \phi^*}\right) + \max\left\{-\frac{\ln \alpha_0 - \ln \bar \alpha}{2\ln \gamma}, 0\right\}$, $C=\min\left\{\frac{1}{(1+\tau)^2},(1-\eta)^2\right\}\frac{\bar\alpha \theta}{4D^2},$
with $p$ and $\bar \alpha$ as defined in Proposition \ref{prop:ass_convex}. 
\end{theorem}

\section{Oracles}\label{oracles}
In this section, we briefly discuss how the first- and zeroth-order oracles used by our framework can be produced in  two common stochastic optimization settings. \rev{For further detail, we refer the readers to \revv{the conference} version of this paper \cite{jin2021high}.} 

\subsection{Expected loss minimization}
%\ml{Maybe take out this part and refer to the neurIPS version}
Let  $\phi(x)=\E_{d \sim \mathcal{D}}[l(x,d)]$, where $x$ is the model parameters, $d$ is a data sample following distribution $\mathcal{D}$, and $l(x,d)$ is the loss when the model parameterized by $x$ is evaluated on data point $d$.

In this case, the zeroth- and first-order oracles are computed by sample averaging over a minibatch $\mathcal{S}$ sampled from $\mathcal{D}$: 
\begin{equation}\label{eq:erm_oracle}
f(x, \mathcal{S}) = \frac{1}{|\mathcal{S}|}\sum_{d\in \mathcal{S}}l(x,d),\quad g(x, \mathcal{S}) = \frac{1}{|\mathcal{S}|}\sum_{d\in \mathcal{S}}\nabla_x l(x,d).
\end{equation}
In general, $\mathcal{S}$ can be chosen to depend on $x$. We now show how our zeroth- and first-order oracle conditions are satisfied by selecting an appropriate 
sample size  $|\mathcal{S}|$. 

\begin{restatable}{prop}{zerothorder}
	\label{prop:zeroth_order}
Let $\hat e(x,d):={l(x,d) - \phi(x)}$ be a $(\hat{\nu}(x), \hat{b}(x))$-subexponential random variable and $\Var_{d\sim \mathcal{D}}\left[l(x,d)\right] \leq \hat{\epsilon}(x)^2$, for some $\hat{\nu}(x), \hat{b}(x), \hat{\epsilon}(x)$. 
Let  $e(x, \mathcal{S}) = \abs{f(x,\mathcal{S}) - \phi(x)}$ and $N=|\mathcal{S}|$,  then
	$$\E_{\mathcal{S}} \left [e(x, \mathcal{S})\right ] \leq \frac{1}{\sqrt{N}}\hat{\epsilon}(x) \quad \text{and} \quad \text{$e(x, \mathcal{S})$ is $(\nu(x), b(x))$-subexponential,}$$
with  $\nu(x) = b(x) = 8e^2\max\left\{\frac{\hat{\nu}(x)}{\sqrt{N}}, \,\hat{b}(x)\right\}$. 
\end{restatable}

\begin{proof}
	See \cite[Appendix A]{jin2021high}.
 \end{proof}

  Thus, $f(x, \mathcal{S})$ is a zeroth-order oracle with $\epsilon_f = \sup_x\frac{1}{\sqrt{N}}\hat{\epsilon}(x)$,  $\nu = \sup_x \nu(x)$, and $b = \sup_x b(x)$, and $\epsilon_f$ can be made arbitrarily small by taking a large enough sample.

Now suppose, for some $M_c, M_v\geq 0$ and for all $x$, 
%$
%	\E_{d\sim \mathcal{D}}\norm{\nabla l(x,d)-\nabla\phi(x)}^2\leq M_c+M_v\norm{\nabla\phi(x)}^2,
%$
\begin{equation}\label{BCN}
	\E_{d\sim \mathcal{D}}\norm{\nabla l(x,d)-\nabla\phi(x)}^2\leq M_c+M_v\norm{\nabla\phi(x)}^2,
\end{equation} 
then $g(x, \mathcal{S})$ is a first-order oracle.   
\begin{restatable}{prop}{firstorder}
	\label{prop:first_order}
Let $g= g(x, \mathcal{S})$. Assuming $\E_{d\sim \mathcal{D}}\nabla l(x,d)=\nabla \phi(x)$, then 
	$$
	\abs{\mathcal{S}} \geq\frac{M_c+M_v\norm{\nabla\phi(x)}^2}{\delta}\min\left\{\frac{1}{\epsilon_g^2} ,\frac{(1+\eta)^2}{\eta^2\norm{\nabla\phi(x)}^2} \right\}.
	$$
% $$\abs{\mathcal{S}}\geq \max\left\{\frac{2M_c}{\delta\epsilon_g^2}, \frac{2M_v{(1+\kappa\alpha)}^2}{\delta \kappa^2\alpha^2}\right\}$$ 
%$$\abs{\mathcal{S}} \geq\frac{M_c+M_v\norm{\nabla\phi(x)}^2}{\delta}\min\left\{\frac{1}{\epsilon_g^2} ,\frac{(1+\kappa\alpha)^2}{\kappa^2\alpha^2\norm{\nabla\phi(x)}^2} \right\}$$
%$$\abs{\mathcal{S}} \geq\min\left\{\frac{M_c+M_v\norm{\nabla\phi(x)}^2}{\delta\epsilon_g^2} ,\frac{(M_c+M_v\norm{\nabla\phi(x)}^2)(1+\kappa\alpha)^2}{\delta\kappa^2\alpha^2\norm{\nabla\phi(x)}^2} \right\}$$
implies 
\[
\P\left(\norm{g - \nabla \phi(x)} \leq \max\{\epsilon_g, \eta\|g\|\}\right) \geq 1-\delta.
\]
\end{restatable}

\begin{proof} 
See \cite[Appendix A]{jin2021high}. 
\end{proof}

Choosing $\eta=\min\{\tau,\kappa\alpha\}$ gives the condition on the sample size $\abs{\mathcal{S}} $ which provides a valid first-order oracle. 
\rev{One can also view the result from the perspective of choosing the sample size first, which dictates the values of $\epsilon_g$, $\tau$ and $\kappa$ that are achieved by the first-order oracle, which in turn determines how far and how fast Algorithm  \ref{alg:ls} will converge. }

%At first glance, this result may not seem useful in practice, because the number of samples that theory requires us to take depends on $M_c$ and $M_v$, which are unknown to the algorithm. However, one can view the result from a different perspective by choosing the sample size first, which dictates the values of $\epsilon_g$, $\tau$ and $\kappa$ that are achieved by the first-order oracle, which in turn determines how far and how fast Algorithm  \ref{alg:ls} will converge. 

%\begin{corollary}
%Suppose we choose the sample size at each iteration to be $\abs{S} = \frac{s_0}{\alpha^2}$, where $s_0$ is a constant and $\alpha$ is the step size parameter. Then for any $\delta > 0$, the resulting gradient estimator is a first-order oracle that is sufficiently accurate with probability $1 - \delta$, and $\epsilon_g, \kappa$ as follows: 
%$$\epsilon_g = \sqrt{\frac{2M_c \amax^2}{\delta s_0}}, \quad 
%\kappa = \frac{1}{\sqrt{\frac{\delta s_0}{2M_v}} - \amax}.$$
%{\color{teal}{with $\tau$:
%$$\epsilon_g= \sqrt{\frac{2M_c}{\delta s_0}}, \quad 
%\tau = \frac{1}{\sqrt{\frac{\delta s_0}{2M_v}} - 1}.$$
%}}
%\end{corollary} 
%

 \subsection{Randomized finite difference gradient approximation}
Gradient estimates based on randomized finite differences using noisy function evaluations have become popular for zeroth-order optimization, particularly for model-free policy optimization in reinforcement learning \cite{salimans2016evolution, fazel2018global}.  

%In this setting, the zeroth order oracle is assumed to be available, 
%but with a more strict assumption that $e(x) \leq \epsilon_f$ deterministically. 
The first-order oracle is obtained using the zeroth-order oracle as follows.
Let $\sU=\{u_i: i=1, \ldots, \abs{\mathcal{U}}\}$  be a set of random vectors, with each vector following some ``nice'' distribution (e.g. standard Gaussian). Then,
\begin{align}		\label{eq:GSG_intro}
	g(x,\sU) =  \sum_{i=1}^{\abs{\mathcal{U}}}\frac{f(x+\sigma u_i,\xi) - f(x,\xi)}{\sigma\abs{\mathcal{U}}}  u_i, 
\end{align}
% or using the central (symmetric or antithetic)  version
%\begin{align}		\label{eq:cGSG_intro}
%	g(x,\Ucal) = \sum_{i=1}^N \frac{f(x+\sigma u_i) - f(x-\sigma u_i)}{2\sigma} \tilde u_i,
%\end{align}
where  $\sigma$ is the {sampling radius}. The proposition below shows that \eqref{eq:GSG_intro}  with a large enough sample size gives a first-order oracle. 
\begin{restatable}{prop}{fdfoo}
	\label{prop:fdfoo}
	%We have
	%
	% $$\E_{u \sim N(0,I)} \norm{g(x, u) - \nabla\phi(x)}^2 \leq M_c + M_v\norm{\nabla\phi(x)}^2,$$
	% where $M_c = 2\left[ \left(\sqrt{n}L\sigma + \frac{\sqrt{n}\epsilon_f}{\sigma} \right)^2 + \frac{3nL^2\sigma^2}{4}(n+2)(n+4) + \frac{12n\epsilon_f^2}{\sigma^2}\right]$ and $M_v = 18n$. 
	Assume that $e(x) \leq \epsilon_f$ deterministically,  for any $x$. Let $g = g(x, \mathcal{U})$, and fix $\epsilon_g = 2 \left(\sqrt{n}L\sigma + \frac{\sqrt{n}\epsilon_f}{\sigma}\right)$ where $n$ is the dimension of $x$. Then 
	$$\abs{\mathcal{U}} \geq \frac{\frac{3}{4}L^2\sigma^2n(n+2)(n+4)+\frac{12\epsilon_f^2}{\sigma^2}n+18n\norm{\nabla\phi(x)}^2}{\delta}\min\left\{\frac{4}{\epsilon_g^2} ,\frac{1}{\left(\frac{\eta}{1+\eta}\norm{\nabla\phi(x)} - \frac{\epsilon_g}{2}\right)^2}\right\}$$
	implies
	\[
	\P\left(\norm{g - \nabla \phi(x)} \leq \max\{\epsilon_g, \eta\|g\|\}\right) \geq 1-\delta.
	\]
	Note that in the setting, $\epsilon_g$ is a fixed bias dependent on $\sigma$, and cannot be made arbitrarily small. 
\end{restatable}
\begin{proof}
See \cite[Appendix A]{jin2021high}. 
\end{proof}

Letting $\eta=\min\{\tau, \kappa\alpha\}$ in the above result provides a valid first-order oracle. 

It is straightforward to relax assumption that   $e(x) \leq \epsilon_f$ for all $x$,  in the above result, by replacing it with a condition $e(x)\leq \epsilon_f^\prime$ and $e(x+\sigma u_i)\leq \epsilon_f^\prime$, for all $u_i\in \sU$, for some $\epsilon_f^\prime>\epsilon_f$. This latter condition happens with high probability dependent of the value on $\epsilon_f^\prime$, due to the properties of the zeroth-order oracle. 

%\begin{corollary}
%\label{cor:fdfoo}
%Setting $\sigma^2 = \frac{4\epsilon_f}{L\sqrt{(d+2)(d+4)}}$, we have
%$$
%\abs{\mathcal{U}} \geq \frac{24L\epsilon_f\sqrt{(d+2)(d+4)}}{\delta\epsilon_g^2} 
%+
%\frac{72d(1+\kappa\alpha)^2}{\delta\kappa^2\alpha^2}
%$$
%	implies
%	\[
%	\P\left(\norm{g - \nabla \phi(x)} \leq \max\{\epsilon_g, \kappa\alpha\|g\|\}\right) \geq 1-\delta.
%	\]
%\end{corollary}
%
%\begin{corollary}
%Suppose we choose the sample size at each iteration to be $\abs{\mathcal{U}} = \frac{s_0}{\alpha^2}$, where $s_0$ is a constant and $\alpha$ is the step size parameter. Then for any $\delta > 0$, the resulting gradient estimator is a first-order oracle that is sufficiently accurate with probability $1 - \delta$, and $\epsilon_g, \kappa$ as follows: 
%$$\epsilon_g = \sqrt{\frac{48 L \epsilon_f \sqrt{(d+2)(d+4)} \amax^2}{\delta s_0}}, \quad 
%\kappa = \frac{1}{\sqrt{\frac{\delta s_0}{144d}} - \amax}.$$
%{\color{teal}{with $\tau$:
%$$\epsilon_g= \sqrt{\frac{48 L \epsilon_f \sqrt{(d+2)(d+4)}}{\delta s_0}}, \quad 
%\tau = \frac{1}{\sqrt{\frac{\delta s_0}{144d}} - 1}.$$
%}}
%\end{corollary}
%
%

\begin{restatable}{remark}{remarkprop}
	%We first note that Proposition \ref{prop:fdfoo} applies only when $\|\nabla\phi(x)\| \geq \frac{1+\kappa\alpha}{\kappa\alpha}\frac{\epsilon_g}{2}$, 
	Note that $\epsilon_g$ defines the neighborhood of convergence for any method that relies on this oracle, and the smallest value for $\epsilon_g$ is achieved by setting
	$\sigma={\cal O}(\sqrt{\epsilon_f})$.   Let us now discuss the minibatch size. Under the assumption that $\epsilon_f$ is small, $\frac{3}{4}L^2\sigma^2n(n+2)(n+4)+\frac{12\epsilon_f^2}{\sigma^2}$ is also small. Thus when $\|\nabla\phi(x)\|$ is larger than or on the order of $\epsilon_g$, then
	the sample set size remains constant and is proportional to $n$.   In \cite{nesterov2017random} a constant step size stochastic gradient descent is applied using sample size $\abs{\mathcal{U}}=1$, thus each step requires about $n$ fewer samples. However, the step size has to be roughly $n$ times smaller to account for the variance of the stochastic oracles based on one sample, thus the overall complexity is the same. 
\end{restatable}
% \begin{remark}
% The way we bound $\E_{u \sim N(0,I)} \norm{g(x, u) - \nabla\phi(x)}^2$ is by using $$\norm{g(x, u) - \nabla\phi(x)}^2 \leq 2\norm{g(x, u) - \nabla F(x)}^2 + 2\norm{\nabla F(x) - \nabla\phi(x)}^2,$$ where $F(x) = \E_{u \sim N(0, I)} [\phi(x+\sigma u)]$ is the Gaussian smoothing  of $\phi$. By \cite{berahas2019theoretical}, we know that $g(x, u)$ is an unbiased estimator of $\nabla F(x)$, and
% \begin{itemize}
% 	\item $\norm{\nabla F(x) - \nabla\phi(x)} \leq \sqrt{n}L\sigma + \frac{\sqrt{n}\epsilon_f}{\sigma}$,  
%	\item $\E_{u \sim N(0, I)} \norm{g(x, u) - \nabla F(x)}^2 \leq 3n\left( 3\norm{\nabla\phi(x)}^2 + \frac{L^2\sigma^2}{4}(n+2)(n+4) + \frac{4\epsilon_f^2}{\sigma^2} \right).$
% \end{itemize} 
% Combining these gives the expressions for $M_c$ and $M_v$ in the proposition. For the complete proof, see the Appendix.  
% \end{remark}
%If 
%$N\geq \sO\left (\frac{n(1+\kappa\alpha )^2}{\delta  \kappa^2\alpha^2 } \right )$
%and $\sigma=\sqrt{ \frac{\epsilon_f}{L}}$, then 
%%for any $ \varepsilon\geq \frac{6n(1+\kappa\alpha)\sqrt{L\epsilon_f}}{\kappa^2\alpha^2}$, 
%the first order oracle conditions hold for $g(x,\sU)$ with $\epsilon_g=\sO(\sqrt{\epsilon_f})$   \cite{berahas2019theoretical}. 

Other finite difference approximation schemes and their centralized versions (see \cite{berahas2019theoretical} for a reference on these) also give suitable first-order oracles. \rev{For brevity, we do not treat them here.}

\section{Experiments}
\label{sec:experiments}
In this section, we illustrate that a step search stochastic algorithm can be efficient in practice. It is important to note we implement the zeroth- and first-order oracles based on fixed minibatch sizes, as is common in stochastic gradient methods, due to implementational considerations. Thus we do not check or ensure that these oracles actually satisfy the properties that our theory requires. 
%We name our algorithm ``ALOE'', which stands for Adaptive Line-search with Oracle Estimations.
The main goal of these experiments is to validate that even a simple implementation of the  method can be competitive with both standard methods such as ADAM \cite{kingma2017adam} and the  ``SGD + Armijo" method proposed in \cite{vaswanie2019painless} and that  while it is important to use $\epsilon_f^\prime>0$ in the step acceptance criterion, estimating this constant is not difficult. A careful exploration of practical variants of SASS, such as heuristics ensuring oracle properties (e.g. adaptive minibatch size) and best choices for  $\gamma$ and $\epsilon_f^\prime$  are subjects for future research.

%Experiments in section \ref{convexEXP} were conducted on a 2020 MacBook Pro with an M1 chip and 16GB of memory, and experiments with Neural networks were conducted on a cluster with GB of memory.

\subsection{Kernel logistic regression} 
\label{sec:convex_exp}
We conduct experiments on all the datasets for binary classification with $150$ to $5000$ data points from the Penn Machine Learning Benchmarks repository (PMLB) \cite{romano2021pmlb}. In total, there are 64 such datasets. Each binary classification problem is formulated as a logistic regression problem with an RBF kernel (with parameter $\sigma = 1$). 

We compare the following three algorithms, each is given a budget of up to $100$ epochs, and they are implemented as follows.

\begin{itemize} 
\item {\bf SASS.} The zeroth- and first-order oracles are implemented using random  mini-batches of size $128$. We estimate $\epsilon_f$ at the beginning of every epoch (i.e. every $K$ iterations, where $K$ equals the total number of data samples divided by $128$), by computing $\frac{1}{5}$ times the empirical standard deviation of $30$  zeroth-order oracle calls with batch size $128$ at the current point. \rev{Figure \ref{epi_f} shows the performance of SASS with different choices of $\epsilon_f^\prime$, with $\epsilon_f^\prime$ being $0$, $\frac 1 5$, $\frac 1 2$, and $1$ times the empirical standard deviation of the zeroth-order oracle. We observe that the algorithm is fairly robust to how $\epsilon_f^\prime$ is chosen in general, as long as it is not chosen to be zero.}
\begin{figure}
	\begin{subfigure}{.33\textwidth}
		\centering
		\includegraphics[trim=20 0 40 0, clip, width=\linewidth]{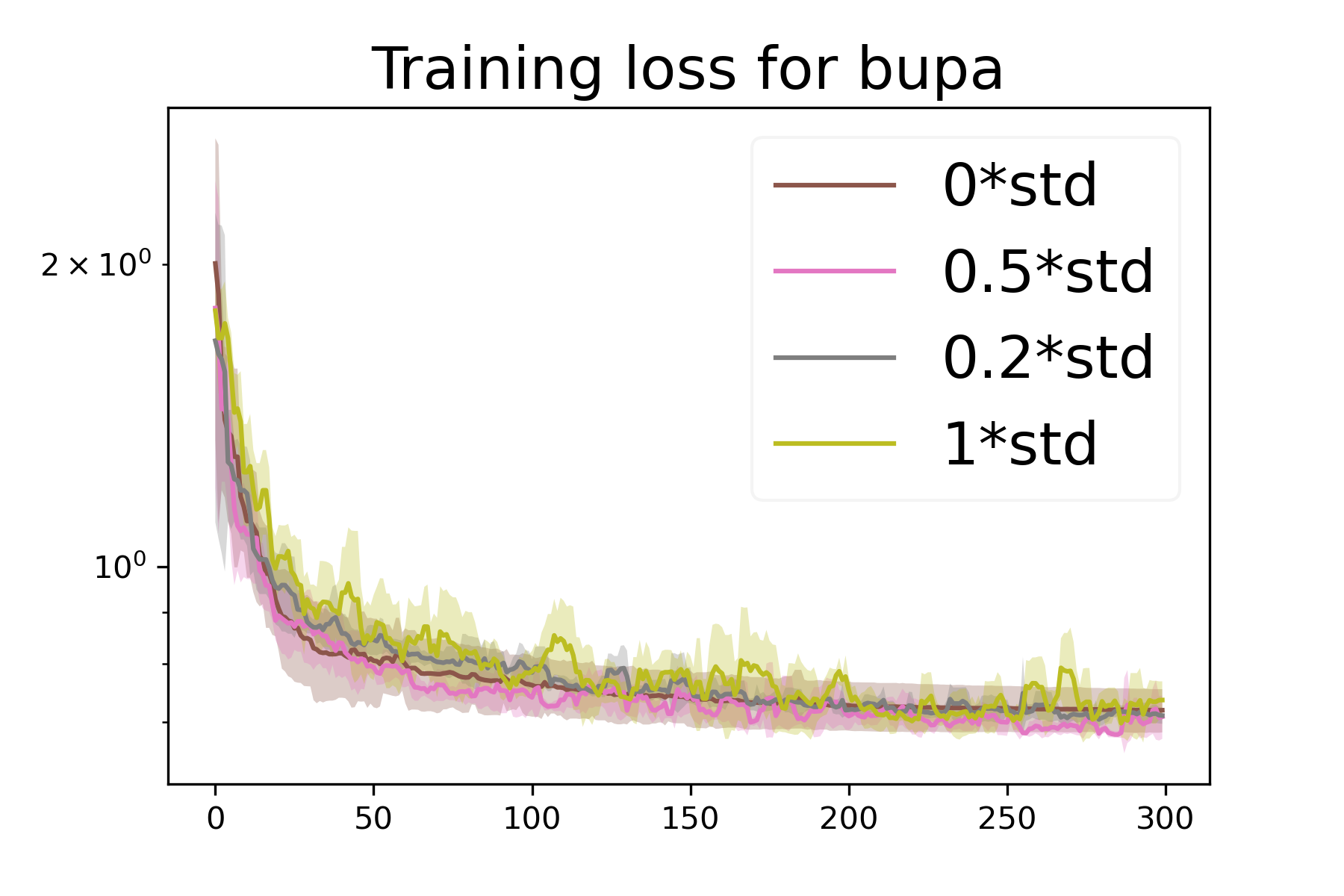}
		%		\caption{Breast (PMLB)}
	\end{subfigure}
	\begin{subfigure}{.33\textwidth}
		\centering
		\includegraphics[trim=8 0 40 0, clip, width=\linewidth]{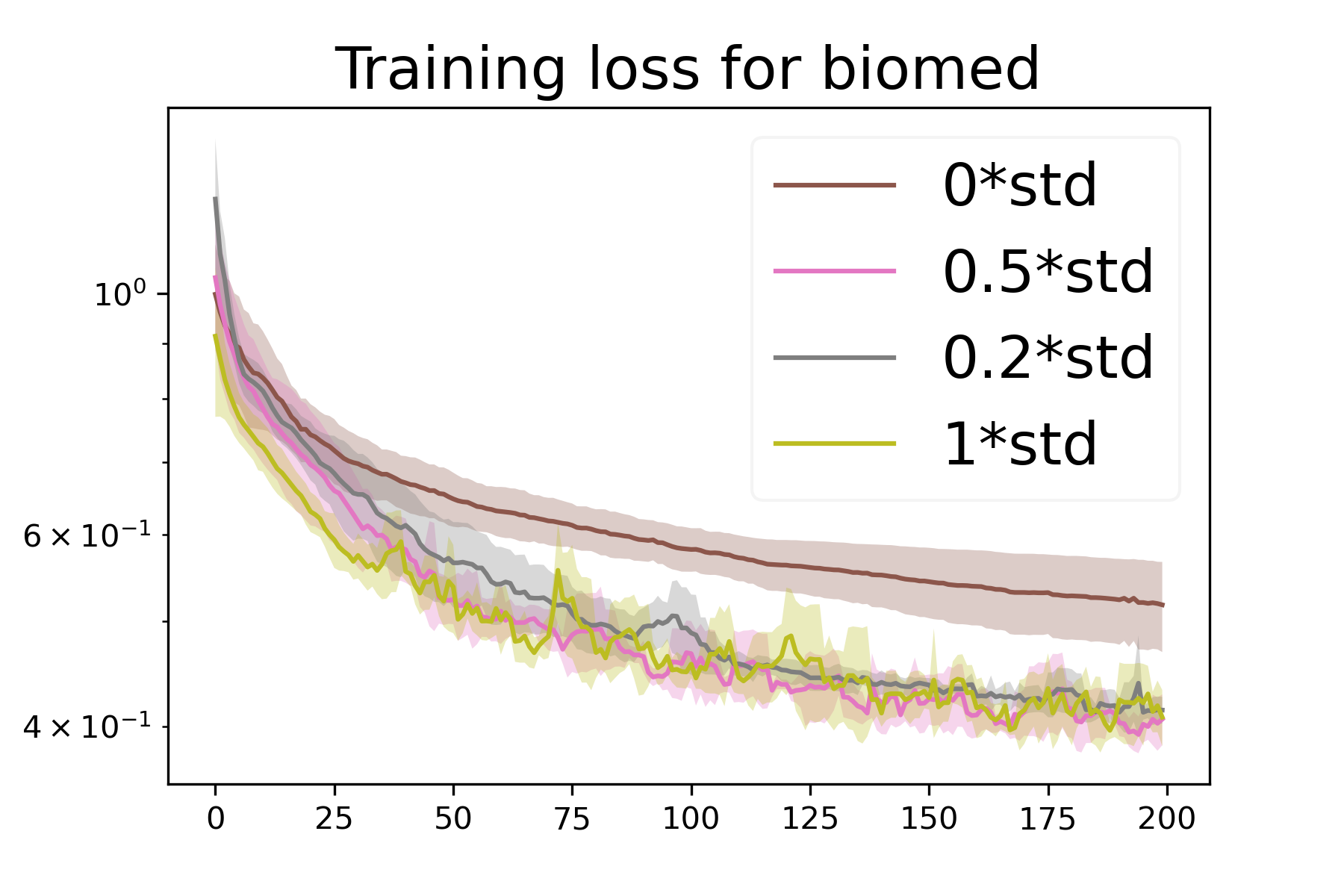}
		%		\caption{Biomed (PMLB)}	
	\end{subfigure}
	\begin{subfigure}{.33\textwidth}
		\centering
		\includegraphics[trim=5 0 40 0, clip, width=\linewidth]{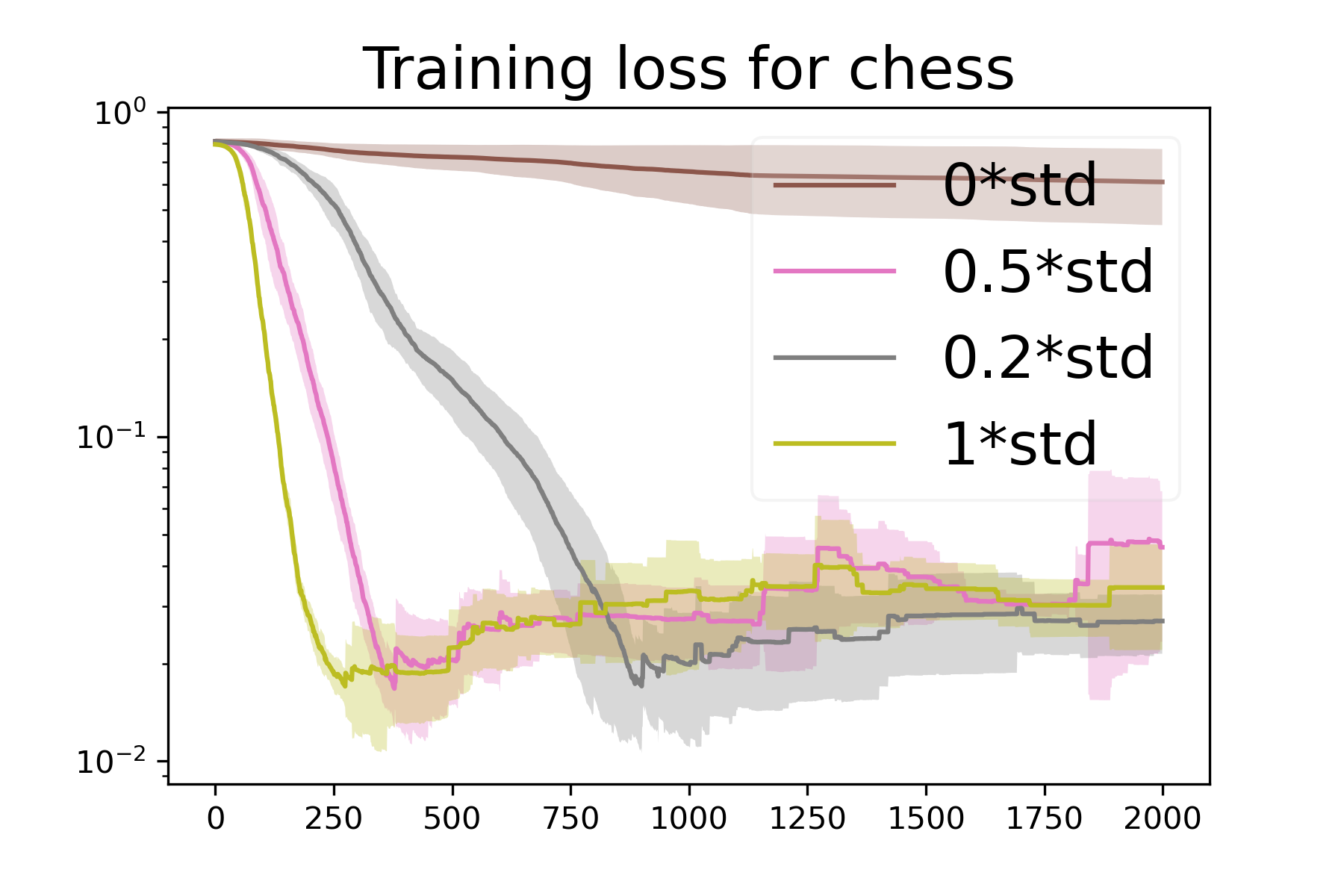}
	\end{subfigure}
	\caption{SASS with different choices of $\epsilon_f^\prime$.}
	\label{epi_f}
\end{figure}
The parameters for all runs were chosen as $\gamma=0.9$, $\theta = 0.2$, $\alpha_0=1$. 

\item {\bf SLS.}  The SLS algorithm (also referred to as  ``SGD + Armijo") proposed in \cite{vaswanie2019painless} differs from SASS in that $\epsilon_f^\prime=0$ and that the same mini-batch is used while backtracking until the Armijo condition is satisfied.  We implemented the algorithm using mini-batch size $128$ and the parameters suggested in Appendix G of their paper.
%, that $\theta=0.1$, $\gamma_{dec}=0.9$, $\gamma_{inc}=1.5$,
%$\eta_{max}=1$ and $\alpha_{max}=10$. 
We tried various parameter combinations for SLS and found the performance of the  suggested parameters to work best.   

\item {\bf ADAM.} ADAM with default parameters as in \cite{adam}, mini-batch size $128$, \rev{with a range of learning rates: $10^{-1}, 10^{-2}, 10^{-3}$ (default), $10^{-4}, \text{ and }10^{-5}$. }
\end{itemize}

We conducted $5$ trials for each dataset and ran each algorithm %for ($100\times$the number of batches in an epoch) iterations,  
with initial points taken randomly from a standard Gaussian distribution. 
\rev{In order to compare the amount of work required by each algorithm more fairly, each algorithm is given the same budget of total work, in terms of the number of inner products. For example, ADAM always runs for $100$ epochs, since it requires $1$ inner product calculation per iteration, while SASS always runs for $50$ epochs, since it requires $2$ inner product calculations per iteration.}  \revv{Note that SASS requires 2 inner products because we use the same batch for the function value estimate and the gradient estimate at the current point $x_k$. Thus it requires one inner product to compute $f(x_k)$ and $g(x_k)$, and another to compute $f(x_k^+)$.} We compare the overall performance of the three algorithms in the following way. 
For each dataset and algorithm, the median best value is defined as the median of the minimum test loss attained over 5 different trials. For each dataset we record the difference between the median best values achieved by SLS vs. SASS. The same is done for ADAM vs. SASS. \rev{Under this metric, SASS achieves better test loss than SLS algorithm in 32 out of $64$ datasets. When compared to ADAM, SASS performed better on test loss (in some cases significantly) for all values of the learning rate except $10^{-1}$. For this learning rate, ADAM performed better than SASS on $47$ out of $64$ datasets. We conclude that SASS is reasonably competitive with these other algorithms, while not requiring step size tuning and having stronger theoretical properties. }

\subsection{MNIST \rev{and Fashion MNIST}}

We now consider non-convex problems. \rev{We train three different neural network architectures, using the softmax loss function}. The first architecture is a multi-layer perceptron (MLP) neural network that has four layers: an input layer with $784$ nodes, two hidden layers with $512$ and $256$ nodes, and an output layer with $10$ nodes. All activation functions are ReLU. This is the same architecture as in \cite{vaswanie2019painless}. The second network is a small convolutional neural network  (CNN) that  in addition to the input and output layers, has two convolutional layers and one fully connected layer. Each convolutional layer uses a $3 \times 3$ kernel with a stride length of $1$, and is followed by a $2 \times 2$ max pooling. This architecture follows the tutorial at \href{https://medium.com/swlh/pytorch-real-step-by-step-implementation-of-cnn-on-mnist-304b7140605a}{this link}\footnote{https://medium.com/swlh/pytorch-real-step-by-step-implementation-of-cnn-on-mnist-304b7140605a}. \rev{The third network is ResNet18 \cite{he2016deep}. We note that the parameter choices and estimation of $\eps_f'$ for SASS are done in the same way as in \Cref{sec:convex_exp}. We tested the MLP and CNN networks on the MNIST dataset \cite{lecun2010mnist}, and ResNet18 on Fashion MNIST \cite{xiao2017fashion}. In Figure \ref{MNIST1}, we plot the results for SASS, SLS, and ADAM with 5 learning rates ($10^{-1}, 10^{-2}, 10^{-3}, 10^{-4}, 10^{-5}$).} 
%Finally, the third network is a medium CNN architecture that in addition to the input and output layers, has three convolutional layers and two fully connected layers. This architecture is the last model in the tutorial  at \href{https://machinelearningmastery.com/how-to-develop-a-convolutional-neural-network-from-scratch-for-mnist-handwritten-digit-classification/}{this link} and the results are shown in Figure \ref{MNIST3}. 

%The algorithms considered are:  
%1) SASS with the default parameters. 2) SLS with the suggested parameters as in the paper \cite{vaswanie2019painless}.  
%3) Adam with the default parameters. 4) SGD with learning rate $10^{-3}$.
%(The reason for using $\gamma=0.9$ in SASS is because in these experiments this value is closer to the parameters chosen by SLS, where $\gamma$ is chosen heuristically, depending on the size of the data set and the mini-batch size.)

In Figure \ref{MNIST1}, the left plots show the progress of the {training} loss of each algorithm, 
%the central plot shows the progress of the {\em testing} loss
the middle plots show the progress of the test loss, and 
%Figure \ref{MNIST1} demonstrates that for  the MLP, the behavior of  ALOE algorithm on the testing and training loss is somewhat similar - both losses start to increase at about the time when the validation accuracies plateau at about  $98\%$. SLS however, like most SGD algorithms, continues to improve the training loss without improving the validation loss. We find this observation supports our theoretical results that ALOE optimizes the expected loss, rather than the empirical loss, but only up to certain accuracy.  
%This shows ALOE is indeed trying to minimize the expected loss of the problem.  
%The average test set error rates for all three algorithms in this case are similar, which are about $2\%$.
%Specifically, the average test set error rates for ALOE $1$, ALOE $2$ and SLS are $0.9\%$, $0.9\%$ and $51\%$ respectively. 
the right plots show the step sizes for SASS and SLS. \rev{The $x$-axis measures the total number of passes (forward and backward) through the neural network, and is a proxy for the total work performed by the algorithm.} \revv{For SASS, we used the same batch for the function value and gradient at the current point as in the convex experiments. Thus each iteration of SASS requires 2 forward passes and 1 backward pass, for a total of 3 passes per iteration. On the other hand, ADAM requires 1 forward pass and 1 backward pass per iteration, for a total of 2 passes per iteration. SLS has a variable number of passes per iteration depending on how many times the algorithm backtracks.}

In these  results, although the oracles are obtained just by using a fixed batch of $128$ data points (which is not adaptive according to the step sizes as suggested by the theory), SASS still works reasonably well, especially on the CNN. We hypothesize that the reason SLS does not perform well on the CNN is because the step size becomes quite small and the algorithm does not manage to progress. SASS however, is able to take large steps and progresses well in this case. Interestingly, SASS performs quite well in terms of the test loss, which represents the expected loss function $\phi$. 
\rev{
For these problems, ADAM performs best with learning rate either $10^{-4}$ or $10^{-5}$, which are different from both the default learning rate ($10^{-3}$) and the best value for the convex problems ($10^{-1}$). }

%Finally,Figure \ref{MNIST3} show good performance of ALOE $2$, but works performance of ALOE $1$, while SLS still performs worse. 
%Specifically, the average test set error rates for ALOE $1$, ALOE $2$ and SLS are $3.9\%$, $0.74\%$ and $7.5\%$ respectively.

\begin{figure}[ht]\centering
	\begin{subfigure}{.32\textwidth}
		\centering
		\includegraphics[trim=3 0 40 0, clip, width=\linewidth]{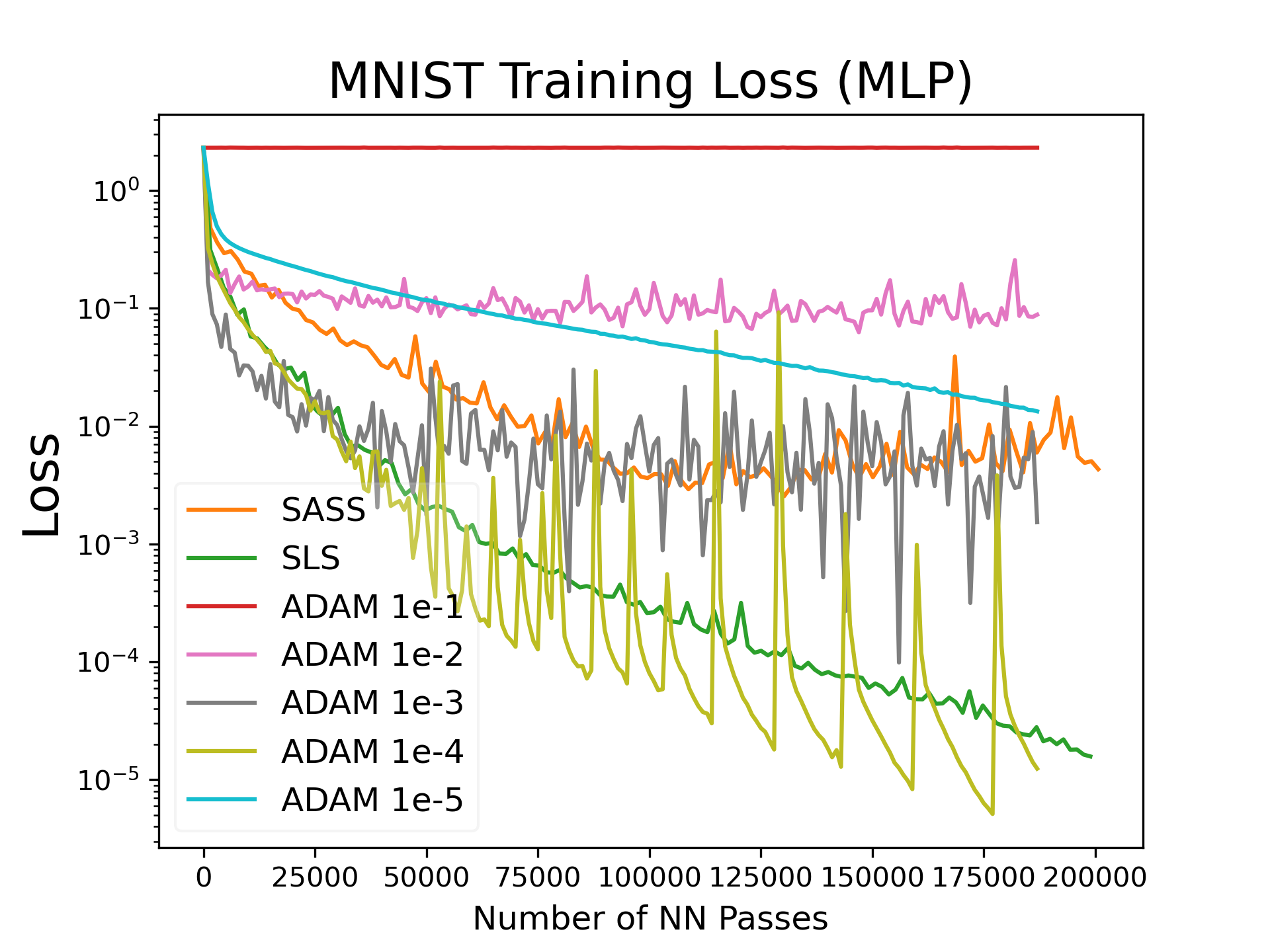}
		%		\caption{Breast (PMLB)}
	\end{subfigure}
	\begin{subfigure}{.32\textwidth}
		\centering
		\includegraphics[trim=3 0 40 0, clip, width=\linewidth]{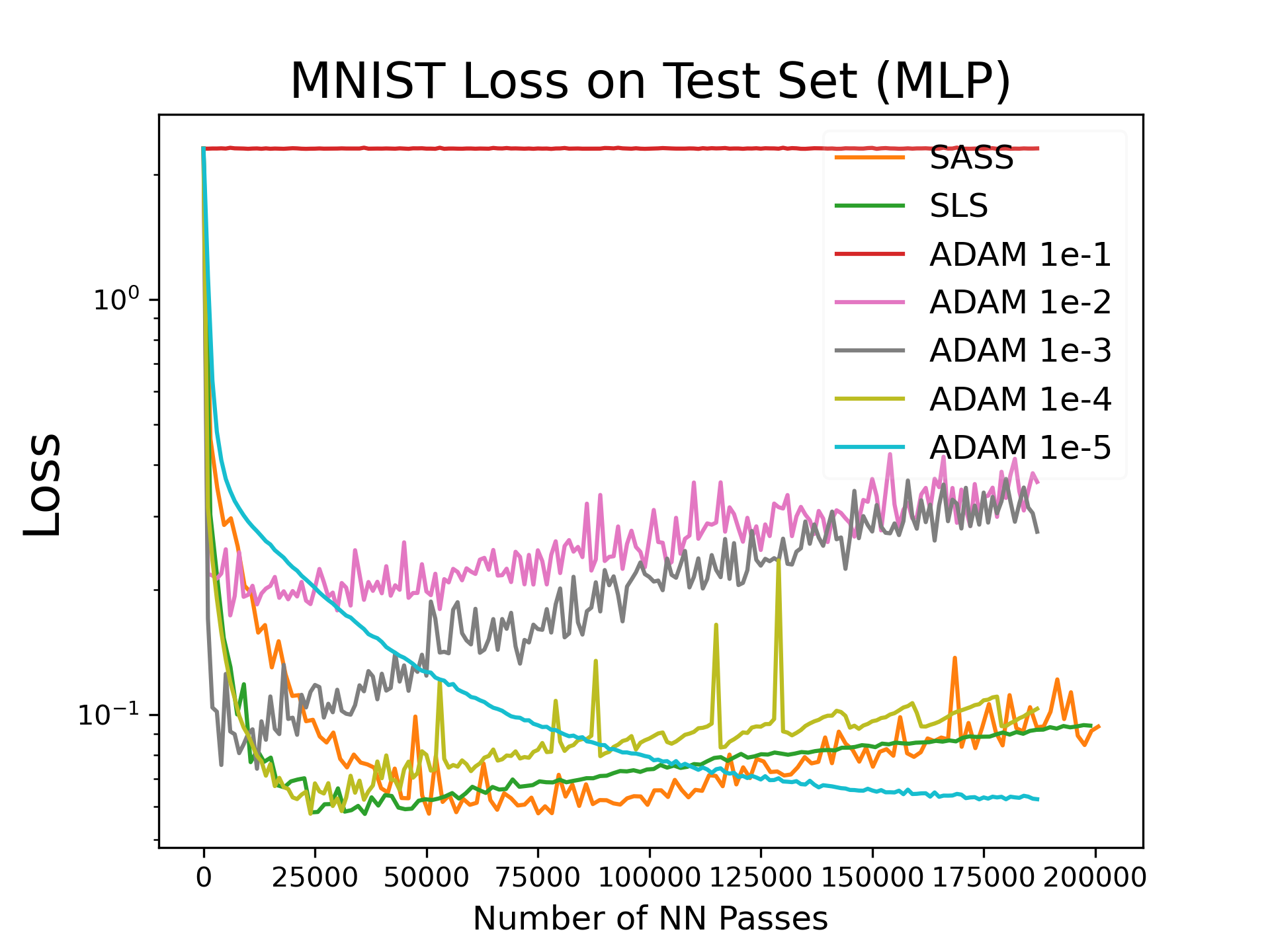}
	\end{subfigure}
	\begin{subfigure}{.32\textwidth}
	\centering
	\includegraphics[trim=3 0 40 0, clip, width=\linewidth]{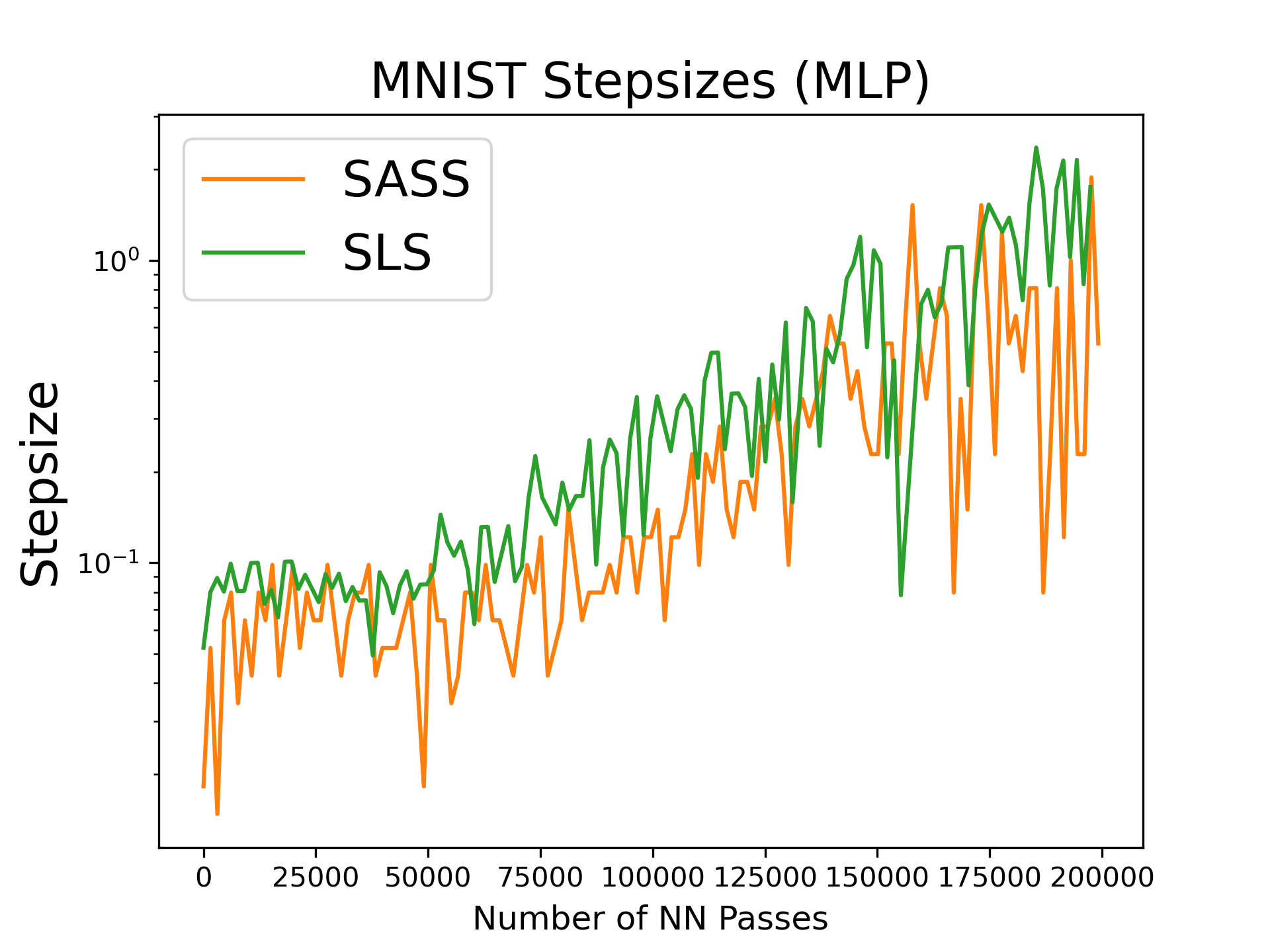}
\end{subfigure}

	\begin{subfigure}{.32\textwidth}
	\centering
	\includegraphics[trim=3 0 40 0, clip, width=\linewidth]{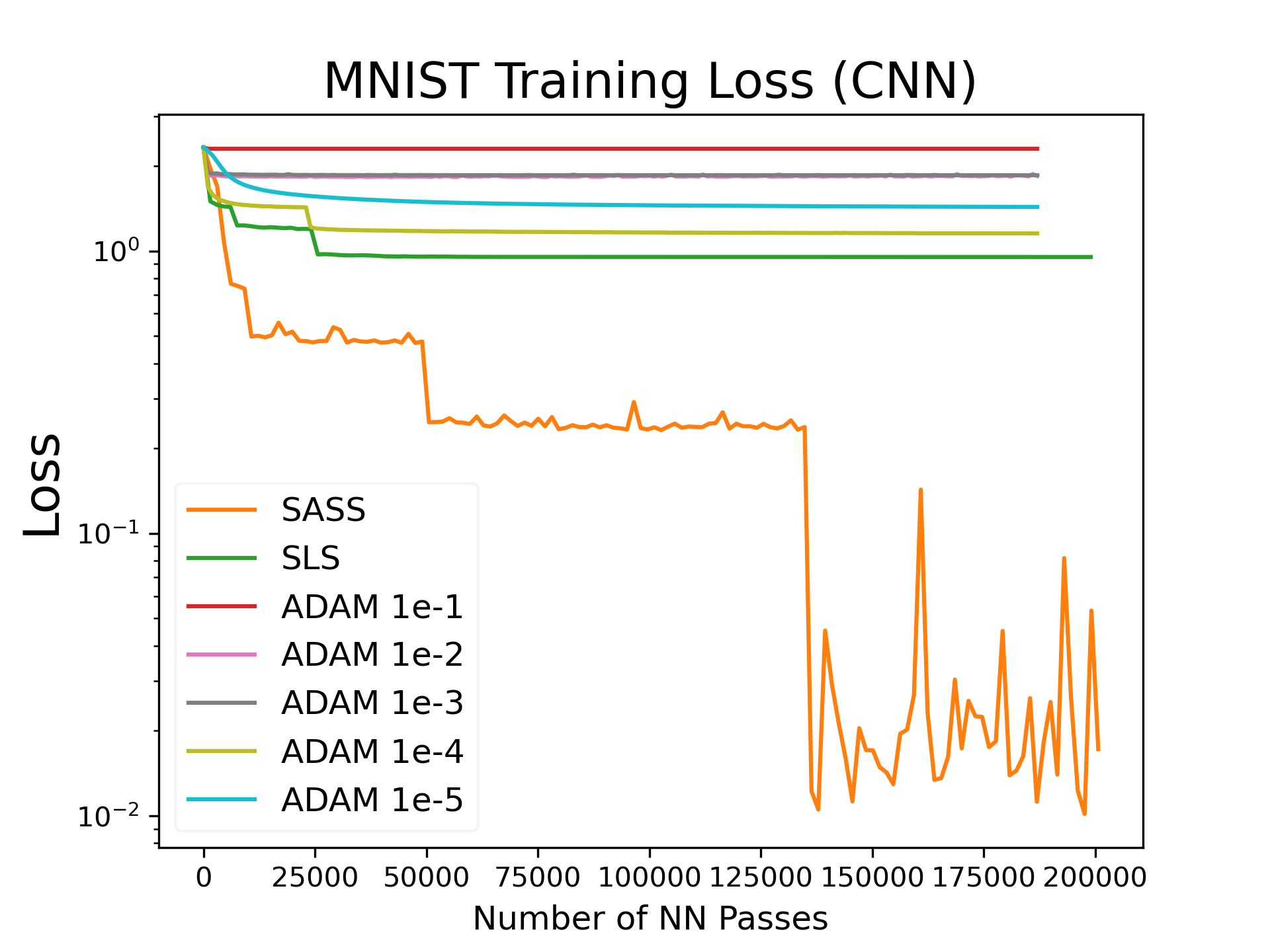}
	%		\caption{Breast (PMLB)}
\end{subfigure}
\begin{subfigure}{.32\textwidth}
	\centering
	\includegraphics[trim=3 0 40 0, clip, width=\linewidth]{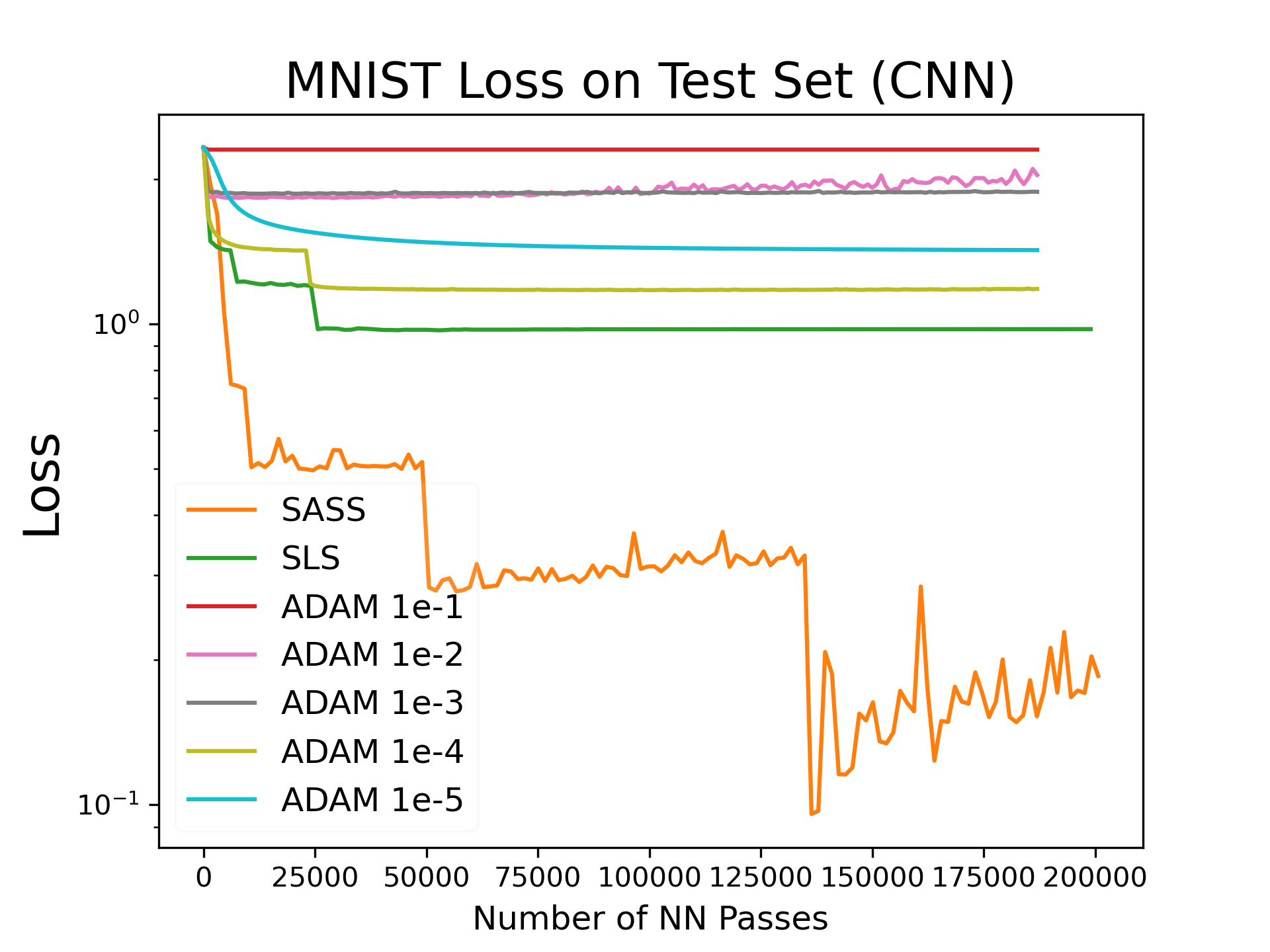}
\end{subfigure}
	\begin{subfigure}{.32\textwidth}
	\centering
	\includegraphics[trim=3 0 40 0, clip, width=\linewidth]{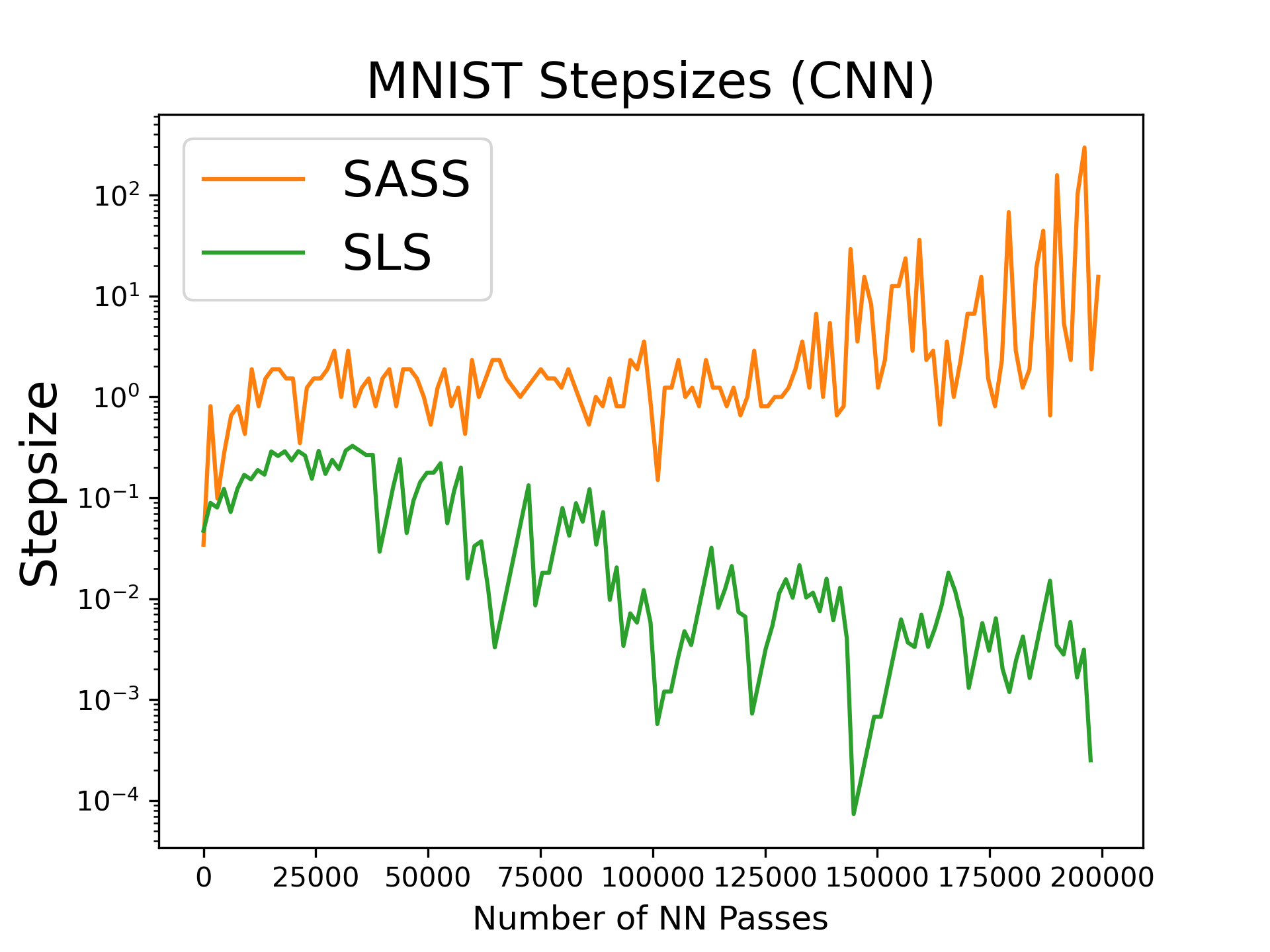}
\end{subfigure}

\begin{subfigure}{.32\textwidth}
	\centering
	\includegraphics[trim=3 0 40 0, clip, width=\linewidth]{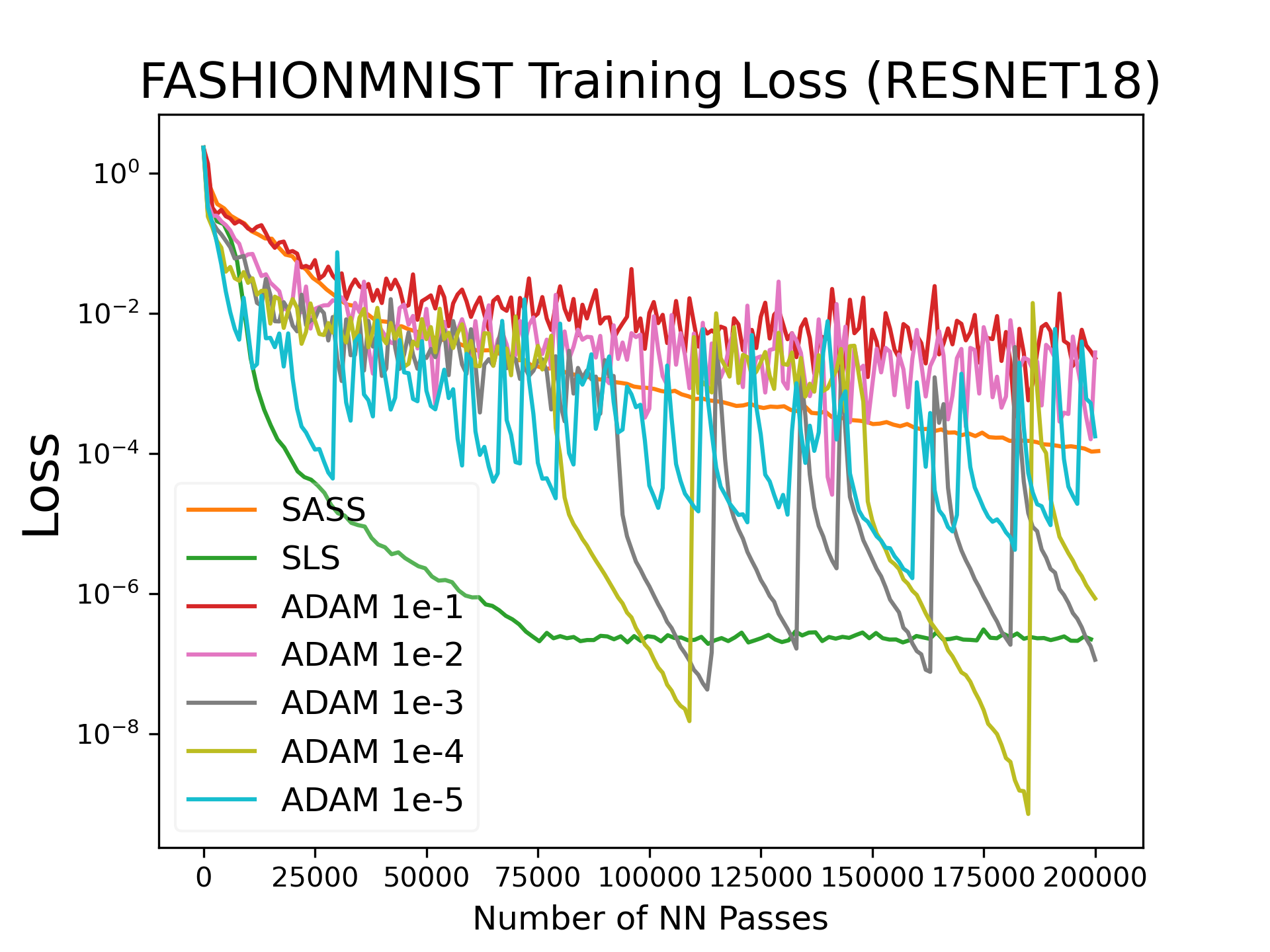}
	%		\caption{Breast (PMLB)}
\end{subfigure}
\begin{subfigure}{.32\textwidth}
	\centering
	\includegraphics[trim=3 0 25 0, clip, width=\linewidth]{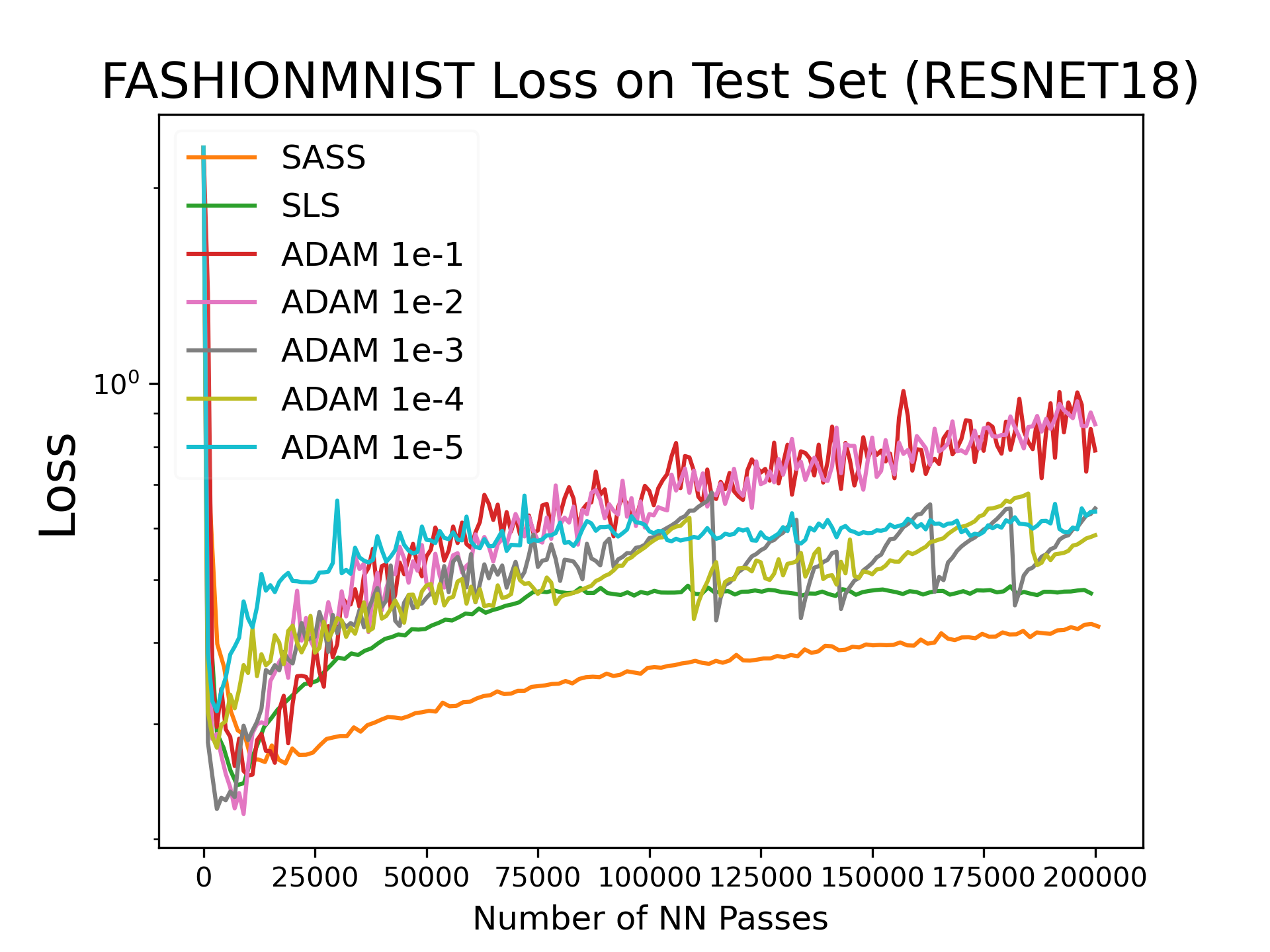}
\end{subfigure}
	\begin{subfigure}{.32\textwidth}
	\centering
	\includegraphics[trim=3 0 40 0, clip, width=\linewidth]{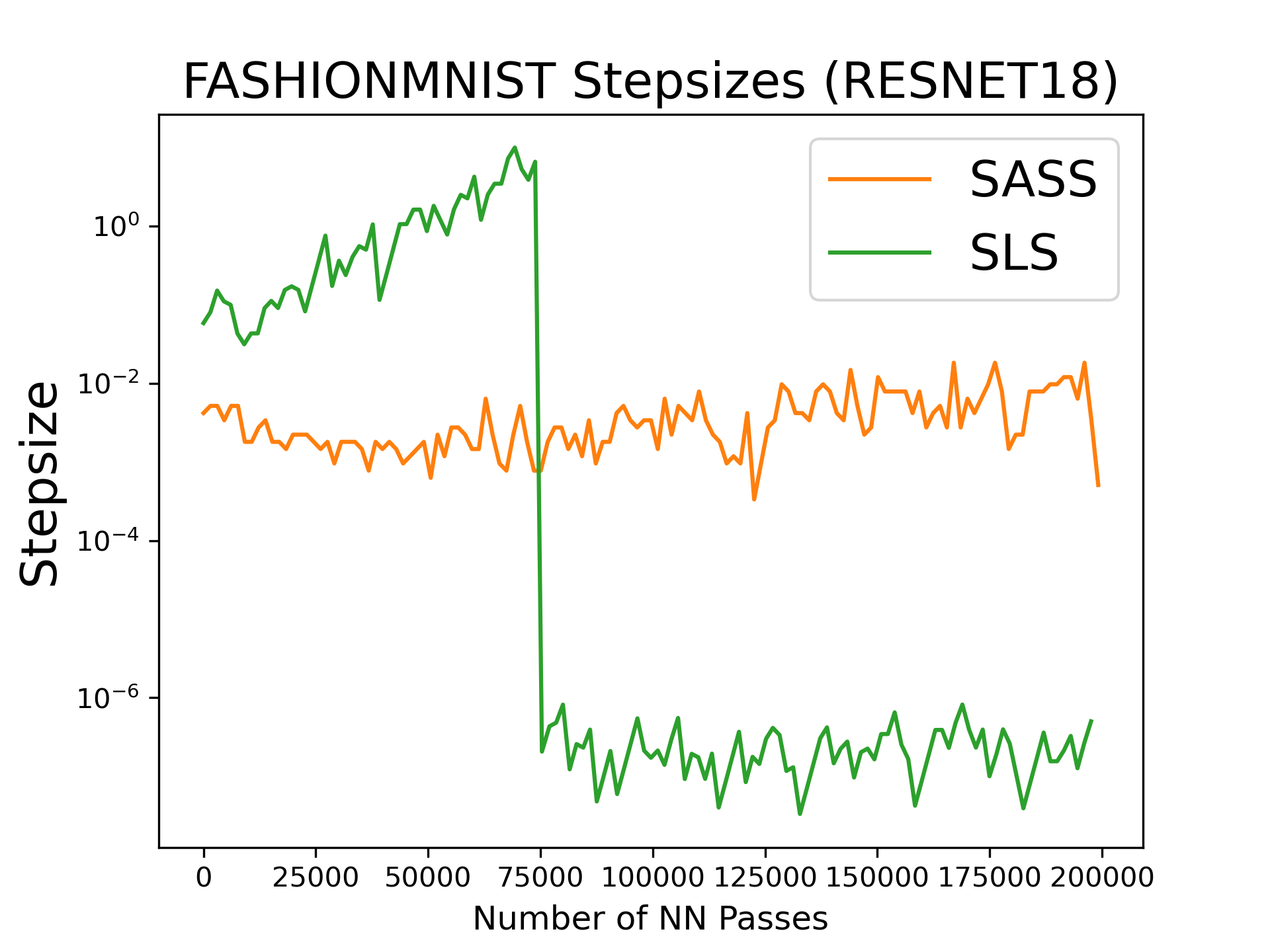}
\end{subfigure}
	\caption{The first two rows shows the results on MNIST with multi-layer perceptron neural network and CNN. The third row is for FashionMNIST with ResNet18.}
	\label{MNIST1}
	
\end{figure}

%
%\begin{figure}[ht]\centering
%
%
%%	\caption{Step-sizes for MNIST with multi-layer perceptron neural network and CNN}
%	\label{MNIST2_step}
%\end{figure}

%\subsubsection{CIFAR-10}
%Finally, we present the experimental results for CIFAR-$10$ dataset on Resnet34 using softmax loss function with the same set of algorithms. 
%For this non-convex problem, we compare the following algorithms:
%1) SASS with the default parameters. 2) SLS algorithm with the suggested parameters as in the paper \cite{vaswanie2019painless}. 3) Adam with the default parameters. 4) SGD with learning rate $10^{-3}$.

%\ml{remove the SGD algorithm in the final plot}
%
%The results are shown in Figure \ref{CIFAR10}. SASS remains competitive for this problem, and achieves similar test error as SLS.
%
%\begin{figure}[ht]
%	\centering
%	\begin{subfigure}{.36\textwidth}
%		\centering
%		\includegraphics[trim=20 0 40 0, clip, width=\linewidth]{cifar10/cifar10_train_loss_resnet34_softmax_loss}
%		%		\caption{Breast (PMLB)}
%	\end{subfigure}
%	\begin{subfigure}{.33\textwidth}
%		\centering
%		\includegraphics[trim=20 0 40 0, clip, width=\linewidth]{cifar10/cifar10_test_loss_resnet34_softmax_loss}
%		%		\caption{Biomed (PMLB)}	
%	\end{subfigure}
%	\begin{subfigure}{.36\textwidth}
%		\centering
%		\includegraphics[trim=3 0 40 0, clip, width=\linewidth]{cifar10/cifar10_val_acc_resnet34_softmax_loss}
%	\end{subfigure}
%	\caption{CIFAR-10 with Resnet34}
%	\label{CIFAR10}
%\end{figure}

\section{Final Remarks}\label{sec:conclusion}
We conclude the paper with a brief overview of our theoretical results in comparison to prior literature. 

In this paper we have substantially extended complexity analysis and relaxed conditions for  step search methods based on stochastic oracles, compared to prior works in   \cite{CS17}, \cite{paquette2018stochastic} and \cite{berahas2019global}. 

The stochastic line search in \cite{vaswanie2019painless} is proposed specifically for empirical risk minimization, and the zeroth- and first-order oracles are implemented using   mini-batch of a fixed size. The same mini-batch is used for all consecutive unsuccessful iterations. This guarantees that a successful iteration  is eventually achieved for Armijo condition with  $\epsilon_f^\prime=0$, under the assumption that  for every mini-batch, $g(x,\xi^\prime)$ is Lipschitz continuous.
	The convergence analysis  then assumes that $M_c=0$ in \eqref{BCN} ({\em strong growth condition}) and in the case when $\phi$ is not convex,
	the step size parameter is bounded above by $\frac{1}{LM_v}$. Thus, the method itself and its convergence are not better than those of a stochastic gradient descent with a fixed step size bounded by $\frac{1}{LM_v}$    \cite{bottou2018optimization}.  
	It is also assumed that the step size is reset to a fixed value at the start of each iteration, which is impractical. Good  computational results are reported in \cite{vaswanie2019painless} for a heuristic version of the algorithm where the restrictions of the step size are removed.
	
	In this paper we analyzed Algorithm \ref{alg:ls} under no restriction on the step size parameter.  
	%For the sake of simplicity of analysis, we assume the step size parameter is reduced and increased by the same multiplicative factor. This can be relaxed to some degree. 
	We also do not assume that  $g(x,\xi^\prime)$ is Lipschitz continuous, we only impose this condition on $\phi$. The cost of relaxing all these assumptions is the use of $\epsilon_f^\prime$. For simplicity of the analysis, $\epsilon_f^\prime$ is assumed to be fixed throughout the algorithm. In practice, it can be re-estimated regularly.  Our experiments show
	that estimating $\epsilon_f^\prime$ is easy and works well in practice. Moreover, one can use much smaller values for $\epsilon_f^\prime$ than theory dictates.

\section*{Acknowledgments}
We thank Jorge Nocedal  and Shigeng Sun for the useful discussion for an earlier version of this paper. This work was partially supported by NSF Grants CCF 20-08434, TRIPODS 17-40796 and ONR award N00014-22-1-215. 

\bibliographystyle{siamplain}
\bibliography{./references.bib}
%	\bibliographystyle{alpha} 
%	\bibliography{references}
	
\appendix

\end{document}